\documentclass[a4paper]{amsart}
\usepackage{amsmath,amsthm,amssymb,color,graphicx}
\usepackage{diagbox}

\usepackage{tikz}

\usepackage{color}
\usepackage[normalem]{ulem}

\newcommand{\Z}{\mathbb{Z}}

\newcommand{\supp}{{\rm supp \,}}

\newtheorem{theorem}{Theorem}
\newtheorem{lemma}[theorem]{Lemma}
\newtheorem{proposition}[theorem]{Proposition}
\newtheorem{remark}[theorem]{Remark}
\newtheorem{example}{Example}

\newtheorem{corollary}[theorem]{Corollary}

\newtheorem{lemma-app}[theorem]{Lemma}
\newtheorem{proposition-app}[theorem]{Proposition}

\newlength{\fullboxwidth}
\DeclareMathAlphabet{\mathpzc}{OT1}{pzc}{m}{it}

\begin{document}

\title{Time-frequency shift invariance of Gabor spaces generated by integer lattices}
\author{Carlos Cabrelli, Dae Gwan Lee, Ursula Molter, G\"otz E.~Pfander}

\address{\textrm{(C.~Cabrelli)}
Departamento de Matem\'atica,
Facultad de Ciencias Exac\-tas y Naturales,
Universidad de Buenos Aires, Ciudad Universitaria, Pabell\'on I,
1428 Buenos Aires, Argentina and
IMAS/CONICET, Consejo Nacional de Investigaciones
Cient\'ificas y T\'ecnicas, Argentina.}
\email{cabrelli@dm.uba.ar}

\address{\textrm{(D.G.~Lee)}
Lehrstuhl f\"ur Mathematik - Wissenschaftliches Rechnen, Mathematisch-Geographische Fakult\"at, Katholische Universit\"at Eichst\"att-Ingolstadt, 85071 Eichst{\"a}tt, Germany.}
\email{daegwans@gmail.com}

\address{\textrm{(U.~Molter)}
Departamento de Matem\'atica,
Facultad de Ciencias Exac\-tas y Naturales,
Universidad de Buenos Aires, Ciudad Universitaria, Pabell\'on I,
1428 Buenos Aires, Argentina and
IMAS/CONICET, Consejo Nacional de Investigaciones
Cient\'ificas y T\'ecnicas, Argentina.}
\email{umolter@dm.uba.ar}

\address{\textrm{(G.E.~Pfander)}
Lehrstuhl f\"ur Mathematik - Wissenschaftliches Rechnen, Mathematisch-Geographische Fakult\"at, Katholische Universit\"at Eichst\"att-Ingolstadt, 85071 Eichst{\"a}tt, Germany.}
\email{pfander@ku.de}

\thanks{The research of
C.~Cabrelli and U.~Molter are partially supported by
Grants  PICT 2011-0436 (ANPCyT), PIP 2008-398 (CONICET).
D.G.~Lee and G.E.~Pfander acknowledge support by the DFG Grants PF 450/6-1 and PF 450/9-1, and would like to thank Andrei Caragea for helpful discussions.
}

%%% \date{}
\begin{abstract}
We study extra time-frequency shift invariance properties of Gabor spaces. For a Gabor space generated by an integer lattice, we state and prove several characterizations for its time-frequency shift invariance with respect to a finer integer lattice.
The extreme cases of full translation invariance, full modulation invariance, and full time-frequency shift invariance are also considered. The results show a close analogy with the extra translation invariance of shift-invariant spaces.
%%% , however, presents subtle but deep differences, due to the non-commutativity of the time-frequency operations.
{\bigskip\newline \noindent} {{\small \textsc{Key words:
Extra time-frequency shift invariance, Gabor space, Time-frequency analysis, Shift-invariant space}}}

{\bigskip\noindent {\small \textsc{AMS Subject Classification
2010: 42C40, 42C15, 46C99}}}
\end{abstract}

\maketitle

\section{Introduction}
\label{sec:intro}

The time-frequency structured systems that are complete in the space of square integrable functions play a fundamental role in applied harmonic analysis. Systems that span a proper subspace are relevant, for example in communications engineering, and many aspects of these have been studied from an application oriented point of view.
From a more mathematical, structure oriented point of view, many aspects remain to be explored.

An interesting question regarding subspaces spanned by time-frequency structured systems is whether they are invariant under time-frequency shifts other than those pertaining to their defining property.
To state the question formally, we define unitary operators, translation $T_u:L^2(\mathbb{R}^d) \rightarrow L^2(\mathbb{R}^d)$, $T_{u} f(x)=f(x-u)$, modulation $M_{\eta}:L^2(\mathbb{R}^d)\rightarrow L^2(\mathbb{R}^d)$, $M_{\eta} f(x) = e^{2\pi i \eta \cdot x} f(x)$, and time-frequency shift $\pi(u,\eta) = M_{\eta} T_u$, where $u, \eta \in \mathbb{R}^d$. For $\varphi \in L^2(\mathbb{R}^d)$ and $\Lambda$ an additive closed subgroup of $\mathbb{R}^{2d}$, we define the time-frequency structured Gabor system $(\varphi, \Lambda) = \{ \pi(u,\eta) \varphi \, : \, (u,\eta) \in \Lambda \}$ and the respective Gabor space $\mathcal{G}(\varphi, \Lambda) = \overline{\mathrm{span}} \{ \pi(u,\eta) \varphi \, : \, (u,\eta) \in \Lambda \}$. Note that, by definition, $\mathcal{G}(\varphi, \Lambda)$ is invariant under time-frequency shift by elements in $\Lambda$, that is, $\pi(u,\eta) f \in \mathcal{G}(\varphi, \Lambda)$ for all $(u,\eta) \in \Lambda$ and $f \in \mathcal{G}(\varphi, \Lambda)$. The question is then, given $(u_0,\eta_0) \notin \Lambda$, what conditions on $\varphi$ are necessary and sufficient for the space $\mathcal{G}(\varphi, \Lambda)$ to be invariant under $\pi(u_0,\eta_0)$?

This question is motivated by the work \cite{ACHKM10} which treats the case of shift-invariant spaces. Extra translation invariance of shift-invariant spaces in $L^2(\mathbb{R}^d)$ is characterized for the single variable case ($d=1$) in \cite{ACHKM10}, and later, for the multivariable case ($d \geq 2$) in \cite{ACP11}.

While only translations are of concern for invariance of shift-invariant spaces, in the case of Gabor spaces one needs to consider translations, modulations and also their combinations (i.e., time-frequency shifts). What makes the invariance properties of Gabor spaces even more difficult to analyze is the fact that time-frequency shifts do not commute. In this paper, we restrict our attention to integer time-frequency lattices in which case all time-frequency shifts do commute.

Some related works are the following.
In \cite{Bow07}, structural properties of Gabor spaces are studied in close analogy with those of shift-invariant spaces. In particular, characterizations for Gabor spaces are given in terms of range functions, analogously to the characterizations for shift-invariant spaces in \cite{bdr2}.
In \cite{CMP16}, time-frequency shift invariance of Gabor spaces is studied in the context of the Amalgam Balian-Low theorem.
The Amalgam Balian-Low Theorem asserts that there is no Gabor system which is a Riesz basis for $L^2(\mathbb{R}^d)$ and at the same time its window function has good time-frequency localization.
As a generalization of this theorem, \cite{CMP16} showed that if a Gabor system generated by a rational lattice and a window function having good decay in time and frequency is a Riesz basis for the Gabor space it spans, then the Gabor space cannot be invariant under time-frequency shifts by elements not in the generating lattice.

In this paper, we mainly focus on extra invariance of Gabor spaces $\mathcal{G}(\varphi, \Lambda)$ where $\varphi \in L^2(\mathbb{R}^d)$ and $\Lambda \subset \mathbb{R}^{2d}$ is an integer lattice, i.e., a lattice contained in $\mathbb{Z}^{2d}$.
When $\Lambda \subsetneq \widetilde{\Lambda} \subseteq \mathbb{Z}^{2d}$, we give complete characterizations for the $\widetilde{\Lambda}$-invariance of $\mathcal{G}(\varphi, \Lambda)$, which turn out to have close analogy with the case for shift-invariant spaces.
A major difference from the shift-invariance space case is that, as often in time-frequency analysis, the Zak transform is employed in place of the Fourier transform. Through the Zak transform, time-frequency shifts are represented on the time-frequency plane and are therefore easier to access than when the Fourier transform is used.
By scaling the Zak transform, the results obtained generalize to the case $\Lambda \subsetneq \widetilde{\Lambda} \subseteq \alpha \mathbb{Z}^d \times \tfrac{1}{\alpha} \mathbb{Z}^d$ where $\alpha > 0$.
We also consider some extreme cases where $\Lambda = \mathbb{Z}^{2d}$ and $\widetilde{\Lambda} = \mathbb{R}^d \times \mathbb{Z}^d$, $\mathbb{Z}^d \times \mathbb{R}^d$, $\mathbb{R}^{2d}$, each of which corresponds to full translation invariance, full modulation invariance, and full time-frequency shift invariance, respectively.

This paper is organized as follows.
Section \ref{sec:prelim} contains some notations and definitions which will be used throughout the paper.
In Section \ref{sec:SIS}, we review some results on extra invariance of shift-invariant spaces.
In Section \ref{sec:Gabor_sp}, motivated from the case for shift-invariant spaces, we state and prove analogous characterizations for extra invariance of Gabor spaces.
An example is given to illustrate our results.

\medskip

\section{Preliminaries}
\label{sec:prelim}

The \emph{Fourier transform} is defined on $L^{1}(\mathbb{R}^d) \cap L^{2}(\mathbb{R}^d)$ by
\begin{align*}
\mathcal{F}[ f ](\xi) = \widehat{f}(\xi) := \int_{\mathbb{R}^d} f(x) \, e^{- 2 \pi i \xi \cdot x} \, dx , \quad f \in L^{1}(\mathbb{R}^d) \cap L^{2}(\mathbb{R}^d) ,
\end{align*}
so that $\mathcal{F}[\cdot]$ extends to a unitary operator from
$L^{2}(\mathbb{R}^d)$ onto $L^{2}(\mathbb{R}^d)$.
The \emph{Zak transform} is densely defined on $L^2 (\mathbb{R}^d)$ by
$$
Z f ( x , \omega ) = \sum_{k \in \mathbb{Z}^d} f (x + k) \, e^{-2 \pi i k \cdot \omega} \;\; \in \; L^2 ([0,1)^{2d}) ,
$$
which is quasi-periodic in the sense that
\begin{align*}
Z f ( x + k , \omega + \ell )
= e^{2 \pi i k \cdot \omega} \, Z f (x, \omega)
\quad \text{for all} \; k , \ell \in \mathbb{Z}^d .
\end{align*}
The mapping $f \mapsto Z f$ is a unitary map from $L^2(\mathbb{R}^d)$ onto $L^2 ([0,1)^{2d})$, where the functions in $L^2 ([0,1)^{2d})$ are understood to be quasi-periodic on $\mathbb{R}^{2d}$.

From the commutation relations $T_{u} M_{\eta} = e^{-2 \pi i u \cdot \eta} \, M_{\eta} T_{u}$, $u, \eta \in \mathbb{R}^d$, we have
\begin{align}
\label{eqn_commut_rel}
\pi(u,\eta) \circ \pi(u',\eta') = e^{2 \pi i (\eta \cdot u' - u \cdot \eta')} \, \pi(u',\eta') \circ \pi(u,\eta)
\quad \text{for} \;\; (u,\eta), (u',\eta') \in \mathbb{R}^d \times \mathbb{R}^d .
\end{align}
For any $u, \eta \in \mathbb{R}^d$ and $f \in L^2(\mathbb{R}^d)$, we have
\begin{align}
\big( Z \pi (u,\eta) f \big) (x, \omega)
&= \sum_{k \in \mathbb{Z}^d} (\pi (u,\eta) f )(x+ k ) \, e^{-2 \pi i k \cdot \omega } = \sum_{k \in \mathbb{Z}^d} e^{2 \pi i \eta \cdot (x + k) } f (x+ k -u) \, e^{-2 \pi i k \cdot \omega } \label{Zak_on_pi} \\
&= e^{2 \pi i \eta \cdot x } \, Z f ( x - u , \omega - \eta) . \nonumber
\end{align}
By the quasi-periodicity of Zak transform, it follows that for $u,\eta \in \mathbb{Z}^d$,
\begin{align}
\label{Zak_on_pi_quasi_peri}
\big( Z \pi (u,\eta) f \big) (x, \omega)
= e^{2 \pi i (\eta \cdot x - u \cdot \omega)} \, Z f ( x , \omega ) .
\end{align}

A \emph{(full rank) lattice} $\Gamma$ in $\mathbb{R}^d$ is a discrete subgroup of $\mathbb{R}^d$ represented by $\Gamma = A \, \mathbb{Z}^d$ for some $A \in GL(d,\mathbb{R})$, where $GL(d,\mathbb{R})$ denotes the general linear group of degree $d$ over $\mathbb{R}$.
We will consider lattices in $\mathbb{R}^d$ for collections of time elements $u \in \mathbb{R}^d$, and lattices in $\mathbb{R}^{2d}$ for collections of time-frequency elements $(u,\eta) \in \mathbb{R}^d \times \mathbb{R}^d$.
We reserve the letter $\Gamma$ for lattices in $\mathbb{R}^d$ and $\Lambda$ for lattices in $\mathbb{R}^{2d}$.
In many cases, \emph{separable lattices} of the form $\Lambda = A \mathbb{Z}^d \times B \mathbb{Z}^d \subset \mathbb{R}^{2d}$, where $A , B \in GL(d,\mathbb{R})$, are considered.
We write $\Lambda = \alpha \mathbb{Z}^d \times \beta \mathbb{Z}^d$ in the case where $A = \alpha I$ and $B = \beta I$, $\alpha, \beta > 0$.

For $\varphi \in L^2(\mathbb{R}^d)$ and an additive closed subgroup $\Lambda \subset \mathbb{R}^{2d}$, let $(\varphi, \Lambda) = \{ \pi(u,\eta) \varphi \, : \, (u,\eta) \in \Lambda \}$ and $\mathcal{G}(\varphi, \Lambda) = \overline{\mathrm{span}} \{ \pi(u,\eta) \varphi \, : \, (u,\eta) \in \Lambda \}$ be the \emph{Gabor system} and \emph{Gabor space}, respectively.
For $\varphi \in L^2(\mathbb{R}^d)$ and an additive closed subgroup $\Gamma \subset \mathbb{R}^d$, let $\mathcal{S}(\varphi, \Gamma) = \mathcal{G}(\varphi, \Gamma \times \{ 0 \}) = \overline{\mathrm{span}} \{ T_{u} \varphi \, : \, u \in \Gamma \}$, in particular, $\mathcal{S}(\varphi, \mathbb{Z}^d)$ is called the \emph{shift-invariant space (SIS)} generated by $\varphi$.

Let $V$ be a closed subspace of $L^2(\mathbb{R}^d)$.
Given $(u,\eta) \in \mathbb{R}^d \times \mathbb{R}^d$, we say that $V$ is \emph{invariant under time-frequency shift by $(u,\eta)$} if $\pi(u,\eta) f \in V$ for all $f \in V$.
Given a subset $\Lambda \subset \mathbb{R}^{2d}$, we say that $V$ is \emph{$\Lambda$-invariant} if $\pi(u,\eta) f \in V$ for all $(u,\eta) \in \Lambda$ and $f \in V$.
Given a subset $\Gamma \subset \mathbb{R}^d$, we say that $V$ is \emph{$\Gamma$-invariant} if it is $\Gamma \times \{0\}$-invariant.
We say that $V$ is \emph{shift-invariant} if it is $\mathbb{Z}^d$-invariant, i.e., $\mathbb{Z}^d \times \{0\}$-invariant.

We define the \emph{time invariance set} of $V$ as
\begin{align*}
\mathcal{T} (V) = \{ u \in \mathbb{R}^d \, : \, T_{u} f \in V ~ \text{for all} ~ f \in V \} .
\end{align*}
If $V$ is shift-invariant, then $\mathcal{T} (V)$ is an additive closed subgroup of $\mathbb{R}^d$ containing $\mathbb{Z}^d$ (Proposition 2.1 in \cite{ACP11}).
Similarly, we define the \emph{time-frequency invariance set} of $V$ as
\begin{align*}
\mathcal{P} (V) = \{ (u,\eta) \in \mathbb{R}^d \times \mathbb{R}^d \, : \, \pi(u,\eta) f \in V ~ \text{for all} ~ f \in V \} .
\end{align*}
If $V$ is $\Lambda$-invariant where $\Lambda \subset \mathbb{R}^{2d}$ is a lattice, then $\mathcal{P} (V)$ is an additive closed subgroup of $\mathbb{R}^{2d}$ containing $\Lambda$ (see Proposition \ref{prop_M_is_closed} in Appendix I).
Thus, if $\mathcal{P} (V)$ contains a lattice $\Lambda \subset \mathbb{R}^{2d}$ and a subset $S \subset \mathbb{R}^{2d}$, then $\mathcal{P} (V)$ contains the smallest additive closed subgroup of $\mathbb{R}^{2d}$ generated by $\Lambda$ and $S$.

\medskip

\section{Shift-Invariant Spaces}
\label{sec:SIS}

As preparation to our analysis on extra invariance of Gabor spaces, we collect some results in shift-invariant spaces.
Extra invariance of shift-invariant spaces in $L^2(\mathbb{R}^d)$ is completely characterized in \cite{ACHKM10} for $d=1$ and in \cite{ACP11} for $d \geq 2$.
We remark that extending single variable results to the multivariate setting is not easily done: the variety of closed subgroups of $\mathbb{R}^d$ for $d \geq 2$ is more complex than in the case $d=1$ where the only possible closed subgroups containing $\Z$, are $\mathbb{R}$ and $\tfrac{1}{n} \mathbb{Z}$, $n \in \mathbb{N}$.

\subsection{Fourier transform characterization of shift-invariant spaces}
\label{subsec:SIS_representation}
$ $

Functions belonging to a shift-invariant space can be characterized using the Fourier transform.
For this we need to recall the notion of dual lattice.
For an additive subgroup $\Gamma$ of $\mathbb{R}^d$,
its \emph{annihilator} is the additive closed subgroup of $\mathbb{R}^d$ given by
$${\Gamma}^{*} = \{ \omega \in \mathbb{R}^d \, : \, e^{- 2 \pi i \gamma \cdot \omega} = 1
\quad \text{for all} \; \gamma \in \Gamma \}.$$
Note that $({\Gamma}^{*})^{*} = \overline{\Gamma}$ (the closure of $\Gamma$ in the standard topology of $\mathbb{R}^d$) and that ${(\Gamma')}^{*} \subset {\Gamma}^{*}$ if $\Gamma \subset \Gamma'$.
If $\Gamma \subset \mathbb{R}^d$ is a (full rank) lattice, then so is ${\Gamma}^{*}$ which is then called the \emph{dual lattice} of $\Gamma$.
If $\Gamma = A \, \mathbb{Z}^d$ where $A \in GL(d,\mathbb{R})$, then ${\Gamma}^{*} = (A^{-1})^T \, \mathbb{Z}^d$.
In particular, ${(c_1 \mathbb{Z} \times \ldots \times c_d \mathbb{Z})}^{*} = \tfrac{1}{c_1} \mathbb{Z} \times \ldots \times \tfrac{1}{c_d} \mathbb{Z}$ where $c_1, \ldots , c_d > 0$.

\begin{lemma}[Theorem 4.3 in \cite{ACP11}]
\label{lem_thm_4_3_in_ACP11}
Let $\varphi \in L^2(\mathbb{R}^d)$ and let $\Gamma$ be an additive closed subgroup of $\mathbb{R}^d$. Then $f \in L^2(\mathbb{R}^d)$ belongs in $\mathcal{S}(\varphi, \Gamma)$ if and only if there exists a ${\Gamma}^{*}$-periodic measurable function $m (\xi)$ such that $\widehat{f}(\xi) = m(\xi) \, \widehat{\varphi}(\xi)$.
\end{lemma}

Note that $\Gamma \subseteq \mathbb{R}^d$ in Lemma \ref{lem_thm_4_3_in_ACP11} is not necessarily discrete.
Lemma \ref{lem_thm_4_3_in_ACP11} was proved in \cite{bdr} for the case where $\Gamma$ is a lattice.

\medskip

\subsection{Extra invariance of shift-invariant spaces}
\label{subsec:SIS_extra_inv}
$ $

While invariance of shift-invariant spaces is concerned with translations only, invariance of Gabor spaces concerns with both translations and modulations.
For this reason, invariance sets associated with shift-invariant spaces and Gabor spaces are subsets of $\mathbb{R}^d$ and $\mathbb{R}^{2d}$ respectively.
To compare these sets, we need to match their ambient space dimensions.
Thus, we will consider shift-invariant spaces in $L^2(\mathbb{R}^{2d})$ and Gabor spaces in $L^2(\mathbb{R}^{d})$ so that their invariance sets are subsets of $\mathbb{R}^{2d}$.

In \cite{ACP11}, extra invariance of shift-invariant spaces in $L^2(\mathbb{R}^{2d})$ is completely characterized, more precisely, the paper characterizes the $\widetilde{\Gamma}$-invariance of shift-invariant spaces where $\widetilde{\Gamma} \subset \mathbb{R}^{2d}$ is an arbitrary closed subgroup containing $\mathbb{Z}^{2d}$. To compare with the case for Gabor spaces, we state the result when $\widetilde{\Gamma} \subset \mathbb{R}^{2d}$ is a (full rank) lattice containing $\mathbb{Z}^{2d}$.

Note that a closed subgroup of $\mathbb{R}^{2d}$ which contains $\mathbb{Z}^{2d}$ and an element in $\mathbb{R}^{2d} \backslash \mathbb{Q}^{2d}$, is non-discrete.
This implies that every lattice $\widetilde{\Gamma} \subset \mathbb{R}^{2d}$ containing $\mathbb{Z}^{2d}$ is a rational lattice, which is a lattice consisting of rational elements only.
In fact, any lattice $\widetilde{\Gamma} \subset \mathbb{R}^{2d}$ containing $\mathbb{Z}^{2d}$ satisfies $\mathbb{Z}^{2d} \subseteq \widetilde{\Gamma} \subseteq \frac{1}{m} \mathbb{Z}^d \times \frac{1}{n} \mathbb{Z}^d$ for some (possibly large) $m, n \in \mathbb{N}$.
Note that its dual lattice $\widetilde{\Gamma}^* \subset \mathbb{R}^{2d}$ satisfies $m \mathbb{Z}^d \times n \mathbb{Z}^d \subseteq \widetilde{\Gamma}^* \subseteq \mathbb{Z}^{2d}$, and $| \mathbb{Z}^{2d} / \widetilde{\Gamma}^* | = | \widetilde{\Gamma} / \mathbb{Z}^{2d} | = | \widetilde{\Gamma} \cap [0,1)^{2d} |$.

\begin{proposition}[\cite{ACHKM10}, \cite{ACP11}]
\label{prop_SIS_gen_inv}
Let $\varphi \in L^2(\mathbb{R}^{2d})$ and let $\widetilde{\Gamma} \subset \mathbb{R}^{2d}$ be a lattice satisfying $\mathbb{Z}^{2d} \subseteq \widetilde{\Gamma} \subseteq \frac{1}{m} \mathbb{Z}^d \times \frac{1}{n} \mathbb{Z}^d$ where $m, n \in \mathbb{N}$ (so that $m \mathbb{Z}^d \times n \mathbb{Z}^d \subseteq \widetilde{\Gamma}^* \subseteq \mathbb{Z}^{2d}$).
We write $\mathbb{Z}^{2d} / \widetilde{\Gamma}^* = \{ I_0 = \widetilde{\Gamma}^* , I_1, \ldots, I_{N-1} \}$, where $N = | \mathbb{Z}^{2d} / \widetilde{\Gamma}^* |$ and the cosets $I_0, I_1, \ldots, I_{N-1}$ form a partition of $\mathbb{Z}^{2d}$.
For $\ell = 0, 1, \ldots , N-1$, let
\begin{align*}
B_{\ell} &= \bigcup_{(r,s) \in I_{\ell}} \, (r,s) + [0,1)^{2d}, \\
U_{\ell} &= \{f \in L^2(\mathbb{R}^d) : \widehat f = \widehat g \cdot \chi_{B_{\ell}} \; \text{ for some } g \in \mathcal{S} (\varphi, \mathbb{Z}^{2d}) \}.
\end{align*}
The following are equivalent.
\begin{itemize}
\item[(a)] $\mathcal{S} (\varphi, \mathbb{Z}^{2d})$ is $\widetilde{\Gamma}$-invariant, that is, $\mathcal{S} (\varphi, \mathbb{Z}^{2d}) = \mathcal{S} (\varphi, \widetilde{\Gamma})$.

\item[(b)]
$U_{\ell} \subseteq \mathcal{S} (\varphi, \mathbb{Z}^{2d})$ for all $\ell = 0, 1, \ldots , N-1$.

\item[(c)]
$\mathcal{F}^{-1}(\widehat \varphi \cdot \chi_{B_{\ell}} ) \subseteq \mathcal{S} (\varphi, \mathbb{Z}^{2d})$ for all $\ell = 0, 1, \ldots , N-1$.

\item[(d)]
For a.e.~$(\xi,\omega)$, $\widehat \varphi (\xi,\omega) \neq 0$ implies that $\widehat \varphi (\xi +r,\omega +s) = 0$ for all $(r,s) \in ( \mathbb{Z}^d \times \mathbb{Z}^d ) \backslash \widetilde{\Gamma}^*$.
Equivalently, for a.e.~$(\xi,\omega)$, at most one of the sums $\sum_{(r,s) \in I_{\ell}} | \widehat{\varphi} (\xi +r,\omega +s) |^2$, $\ell = 0, 1, \ldots , N-1$ is nonzero.
\end{itemize}
Moreover, if any one of the above holds, $\mathcal{S} (\varphi, \mathbb{Z}^{2d})$ is the orthogonal direct sum
$$\mathcal{S} (\varphi, \mathbb{Z}^{2d}) = U_{0} \oplus \cdots \oplus U_{N-1}$$
with each $U_{\ell}$ being a (possibly trivial) subspace of $\mathcal{S} (\varphi, \mathbb{Z}^{2d})$ which is invariant under translations by $\widetilde{\Gamma}$.
\end{proposition}

From the fact that $\mathcal{S} (\varphi, \mathbb{Z}^{2d})$ is translation invariant if and only if it is $\tfrac{1}{m} \mathbb{Z}^d \times \tfrac{1}{n} \mathbb{Z}^d$-invariant for all $m,n \in \mathbb{N}$, we obtain the following.

\begin{proposition}
\label{prop_SIS_trans_inv}
Let $\varphi \in L^2(\mathbb{R}^{2d})$. Then $\mathcal{S} (\varphi, \mathbb{Z}^{2d})$ is invariant under all translations if and only if $\widehat{\varphi} (\xi,\omega)$ vanishes a.e.~outside a fundamental domain of the lattice $\mathbb{Z}^{2d}$.
\end{proposition}

\begin{remark}
\label{rmk_SIS_extra_inv}
\rm
Proposition \ref{prop_SIS_gen_inv} hinges on the representations associated with $\mathcal{S} (\varphi, \mathbb{Z}^{2d})$ and $\mathcal{S} (\varphi, \widetilde{\Gamma})$.
If $\mathcal{S} (\varphi, \mathbb{Z}^{2d}) = \mathcal{S} (\varphi, \widetilde{\Gamma})$ where $\varphi \in L^2(\mathbb{R}^{2d})$ and $\widetilde{\Gamma} \supsetneq \mathbb{Z}^{2d}$, then every function $f$ in $\mathcal{S} (\varphi, \mathbb{Z}^{2d})$ can be expressed in two different ways (in the Fourier transform domain):
\begin{align*}
m(\xi,\omega) \, \widehat{\varphi}(\xi,\omega) = \widehat{f}(\xi,\omega) = \widetilde{m}(\xi,\omega) \, \widehat{\varphi}(\xi,\omega)
\quad \text{a.e.} ,
\end{align*}
where $m(\xi,\omega)$ is $\mathbb{Z}^{2d}$-periodic and $\widetilde{m}(\xi,\omega)$ is $\widetilde{\Gamma}^*$-periodic, and thus we have
\begin{align*}
m(\xi,\omega) = \widetilde{m}(\xi,\omega)
\quad \text{for a.e.} \; (\xi,\omega) \;\; \text{such that} \; \widehat{\varphi} (\xi,\omega) \neq 0.
\end{align*}
Picking $\widetilde{m}(\xi,\omega)$ a genuinely $\widetilde{\Gamma}^*$-periodic function (e.g., $\widetilde{m}(\xi,\omega) = e^{-2 \pi i (x \cdot \frac{a}{m} + \omega \cdot \frac{b}{n})}$ if $f = T_{(\frac{a}{m},\frac{b}{n})} \varphi$ and $(\frac{a}{m},\frac{b}{n}) \in \widetilde{\Gamma}$ for some $a,b \in \mathbb{Z}$) and exploiting the fact that $\widetilde{\Gamma}^* \subsetneq \mathbb{Z}^{2d}$, we get some restrictions on set $\{ (\xi,\omega) \, : \, \widehat{\varphi} (\xi,\omega) \neq 0 \}$ which is defined up to a measure zero set.
Clearly, it is impossible that $\widehat{\varphi} (\xi,\omega) \neq 0$ a.e.
This yields the condition (d) in Proposition \ref{prop_SIS_gen_inv}.
\end{remark}

\medskip

\section{Gabor spaces}
\label{sec:Gabor_sp}

When considering time-frequency shift invariant spaces, i.e., Gabor spaces, the Zak transform replaces the Fourier transform and adjoint lattice takes over the role of dual lattice (compare Lemma \ref{lem_thm_4_3_in_ACP11} with Lemma \ref{lem_Gabor_ftn_exp_lambda_integer}).

\subsection{Zak transform representation for Gabor spaces.}
\label{subsec:Gabor_representation}
$ $

Recall that Lemma \ref{lem_thm_4_3_in_ACP11} gives Fourier transform representation for shift-invariant spaces. In this section, we treat analogous representations for Gabor spaces using Zak transform.

For a (full rank) lattice $\Lambda \subset \mathbb{R}^{2d}$, its \emph{adjoint lattice} is defined by
$$
{\Lambda}^{\circ} = \{ (x,\omega) \in \mathbb{R}^{2d} \, : \, \pi (u,\eta) \circ \pi (x,\omega) = \pi (x,\omega) \circ \pi (u,\eta)
\quad \text{for all} \; (u,\eta) \in \Lambda \} .
$$
Using the relation (\ref{eqn_commut_rel}), we immediately see that
$$
{\Lambda}^{\circ}
= \{ (x,\omega) \in \mathbb{R}^{2d} \, : \, e^{2 \pi i (\eta \cdot x - u \cdot \omega)} = 1
\quad \text{for all} \; (u,\eta) \in \Lambda \} .
$$
If $\Lambda = A \, \mathbb{Z}^{2d}$ where $A \in GL(2d,\mathbb{R})$, then
\begin{align}
\label{eqn_adjoint_lattice_gen_matrix_form}
{\Lambda}^{\circ} = \left( {0 \atop -I_d} \; {I_d \atop 0} \right) (A^{-1})^T \, \mathbb{Z}^{2d} .
\end{align}
If $\Lambda$ is a separable lattice of the form $\Lambda = A \, \mathbb{Z}^d \times B \, \mathbb{Z}^d$ where $A, B \in GL(d,\mathbb{R})$, then ${\Lambda}^{\circ} = (B^{-1})^T \, \mathbb{Z}^d \times (A^{-1})^T \, \mathbb{Z}^d$ (cf.~\cite[p.154]{FZ98}).
In particular, $\left( \alpha \mathbb{Z}^d \times \beta \mathbb{Z}^d \right)^{\circ} = \tfrac{1}{\beta} \mathbb{Z}^d \times \tfrac{1}{\alpha} \mathbb{Z}^d$ where $\alpha, \beta > 0$.
It is easily seen that $({\Lambda}^{\circ})^{\circ} = \Lambda$ for any lattice $\Lambda \subset \mathbb{R}^{2d}$, and that the adjoint reverses the inclusions: ${(\Lambda')}^{\circ} \subset {\Lambda}^{\circ}$ if $\Lambda \subset \Lambda'$.

When $\Lambda \subseteq \mathbb{Z}^{2d}$, we have ${\Lambda}^{\circ} \supseteq (\mathbb{Z}^{2d})^{\circ} = \mathbb{Z}^{2d}$ and in this case the functions in $\mathcal{G}(\varphi, \Lambda)$ are accessible through a simple expression using the Zak transform.

\begin{lemma}
\label{lem_Gabor_ftn_exp_lambda_integer}
Let $\varphi \in L^2(\mathbb{R}^d)$ and let $\Lambda \subseteq \mathbb{Z}^{2d}$ be a lattice.
Then $f \in L^2(\mathbb{R}^d)$ belongs to $\mathcal{G}(\varphi, \Lambda)$ if and only if
there exists a ${\Lambda}^{\circ}$-periodic measurable function $h ( x , \omega )$ such that
\begin{align}
\label{eqn_Zf_equal_h_Zvarphi}
Z f ( x , \omega ) = h ( x , \omega ) \,  Z \varphi ( x , \omega ) .
\end{align}
\end{lemma}

%%%%% A detailed proof of Lemma \ref{lem_Gabor_ftn_exp_lambda_integer} is given in Appendix II. Below we describe the main mechanics of Lemma \ref{lem_Gabor_ftn_exp_lambda_integer}.
A proof of Lemma \ref{lem_Gabor_ftn_exp_lambda_integer} is given in Appendix II.
Below we describe the main mechanics of the proof, to help the reader understand the following results.
Assume that $(\varphi, \Lambda)$ is a frame for its closed linear span $\mathcal{G}(\varphi, \Lambda)$, so that every $f \in \mathcal{G}(\varphi, \Lambda)$ can be expressed in the form
\begin{align*}
f = \sum_{(u,\eta) \in \Lambda}  c_{u,\eta} \, \pi (u,\eta) \, \varphi ,
\quad  \{ c_{u,\eta} \}_{(u,\eta) \in \Lambda} \in \ell^2 (\Lambda) .
\end{align*}
Applying the Zak transform on both sides and using (\ref{Zak_on_pi_quasi_peri}), we obtain the equation (\ref{eqn_Zf_equal_h_Zvarphi}) with $h ( x , \omega ) = \sum_{(u,\eta) \in \Lambda}  c_{u,\eta} \, e^{2 \pi i (\eta \cdot x - u \cdot \omega)} \in L^{2}_{loc} (\mathbb{R}^d \times \mathbb{R}^d)$.
Note that the requirement $\Lambda \subseteq \mathbb{Z}^{2d}$ enables the use of (\ref{Zak_on_pi_quasi_peri}), and that $h ( x , \omega )$ is ${\Lambda}^{\circ}$-periodic, since for any $(x_0,\omega_0) \in {\Lambda}^{\circ}$,
\begin{align*}
h ( x + x_0, \omega + \omega_0)
&= \sum_{(u,\eta) \in \Lambda}  c_{u,\eta} \, e^{2 \pi i (\eta \cdot x - u \cdot \omega)} \, e^{2 \pi i (\eta \cdot x_0 - u \cdot \omega_0)}
= \sum_{(u,\eta) \in \Lambda}  c_{u,\eta} \, e^{2 \pi i (\eta \cdot x - u \cdot \omega)} \cdot 1
=  h ( x , \omega ) .
\end{align*}

As can be seen above, the condition $\Lambda \subseteq \mathbb{Z}^{2d}$ plays a crucial role in Lemma \ref{lem_Gabor_ftn_exp_lambda_integer} and therefore cannot be dropped.
Conversely, assume that (\ref{eqn_Zf_equal_h_Zvarphi}) holds for some ${\Lambda}^{\circ}$-periodic measurable function $h ( x , \omega )$ where $\Lambda \subset \mathbb{R}^{2d}$ is a lattice. Then since $Z f ( x , \omega )$ and $Z \varphi ( x , \omega )$ are quasi-periodic, $h ( x , \omega )$ can be replaced with a function which is both ${\Lambda}^{\circ}$-periodic and $\mathbb{Z}^{2d}$-periodic.
That is, $h ( x , \omega )$ can be always assumed to be $\mathbb{Z}^{2d}$-periodic, which naturally suggests that ${\Lambda}^{\circ} \supseteq \mathbb{Z}^{2d}$, i.e., $\Lambda \subseteq \mathbb{Z}^{2d}$.
Hence, the requirement $\Lambda \subseteq \mathbb{Z}^{2d}$ in Lemma \ref{lem_Gabor_ftn_exp_lambda_integer} is not only essential but also very natural for (\ref{eqn_Zf_equal_h_Zvarphi}) to hold.

Note that since both sides of (\ref{eqn_Zf_equal_h_Zvarphi}) are quasi-periodic, it is sufficient to check the equality (\ref{eqn_Zf_equal_h_Zvarphi}) only for a.e.~$(x,\omega)$ in $[0,1)^{2d}$.

\medskip

\subsection{Extra time-frequency shift invariance of Gabor spaces.}
\label{subsec:Gabor_extra_inv}
$ $

Equipped with the representation for Gabor spaces, we are ready to analyze extra invariance of Gabor spaces $\mathcal{G}(\varphi, \Lambda)$ where $\varphi \in L^2(\mathbb{R}^d)$ and $\Lambda \subseteq \mathbb{Z}^{2d}$ is a lattice.

Let $\widetilde{\Lambda} \subseteq \mathbb{R}^{2d}$ be a closed subgroup which contains $\Lambda$ strictly, that is, $\Lambda \subsetneq \widetilde{\Lambda} \subseteq \mathbb{R}^{2d}$.
Then $\mathcal{G}(\varphi, \Lambda)$ is $\widetilde{\Lambda}$-invariant if and only if $\mathcal{G}(\varphi, \Lambda) = \mathcal{G}(\varphi, \widetilde{\Lambda})$, in which case every $f \in \mathcal{G}(\varphi, \Lambda)$ admits another representation as a function of $\mathcal{G}(\varphi, \widetilde{\Lambda})$.

\subsubsection{The case $\Lambda \subseteq \widetilde{\Lambda} \subseteq \mathbb{Z}^{2d}$.}
\label{subsubsec:extra_inv_subset_int}
$ $

As our first main result, we characterize the $\widetilde{\Lambda}$-invariance of $\mathcal{G}(\varphi, \Lambda)$ when $\Lambda, \widetilde{\Lambda} \subseteq \mathbb{R}^{2d}$ are lattices such that $\Lambda \subseteq \widetilde{\Lambda} \subseteq \mathbb{Z}^{2d}$.

\begin{theorem}
\label{thm_Gabor_integer_lattice}
Let $\varphi \in L^2(\mathbb{R}^d)$ and let $\Lambda, \widetilde{\Lambda} \subseteq \mathbb{R}^{2d}$ be lattices satisfying $\Lambda \subseteq \widetilde{\Lambda} \subseteq \mathbb{Z}^{2d}$ (so that ${\Lambda}^{\circ} \supseteq \widetilde{\Lambda}^{\circ} \supseteq \mathbb{Z}^{2d}$).
We write the quotient ${\Lambda}^{\circ} / \widetilde{\Lambda}^{\circ}$ as $\{ I^{(0)} = \widetilde{\Lambda}^{\circ} , I^{(1)}, \ldots, I^{(N-1)} \}$, where $N$ is the order of ${\Lambda}^{\circ} / \widetilde{\Lambda}^{\circ}$ and the cosets $I^{(0)}, I^{(1)}, \ldots, I^{(N-1)}$ all together forms a partition of ${\Lambda}^{\circ}$.
Let $D \subset [0,1)^{2d}$ be a fundamental domain of the lattice ${\Lambda}^{\circ}$.
For $\ell = 0, 1, \ldots , N-1$, let
\begin{align*}
B^{(\ell)} &= \bigcup_{(u,\eta) \in I^{(\ell)}} (u,\eta) + D , \\
U^{(\ell)} &= \{f \in L^2(\mathbb{R}^d) : Z f = Z g \cdot \chi_{B^{(\ell)}} \; \text{ for some } g \in \mathcal{G}(\varphi, \Lambda) \}.
\end{align*}
The following are equivalent.
\begin{itemize}
\item[(a)] $\mathcal{G}(\varphi, \Lambda)$ is $\widetilde{\Lambda}$-invariant, i.e., $\mathcal{G}(\varphi, \Lambda) = \mathcal{G}(\varphi, \widetilde{\Lambda})$.

\item[(b)]
$U^{(\ell)} \subseteq \mathcal{G}(\varphi, \Lambda)$ for all $\ell = 0, 1, \ldots , N-1$.

\item[(c)]
$Z^{-1} ( Z \varphi \cdot \chi_{B^{(\ell)}} ) \in \mathcal{G}(\varphi, \Lambda)$ for all $\ell = 0, 1, \ldots , N-1$.

\item[(d)]
For a.e.~$(x,\omega)$,
\begin{align}
\label{eqn_cond_d_of_thm_Gabor_integer_lattice}
Z \varphi (x , \omega ) \neq 0 \quad \text{implies}
\quad Z \varphi ( x + u , \omega + \eta ) = 0
\quad \text{for all} \;(u,\eta) \in {\Lambda}^{\circ} \backslash \widetilde{\Lambda}^{\circ} .
\end{align}
Equivalently, for a.e.~$(x,\omega)$, at most one of the sums $\sum_{(u,\eta) \in I^{(\ell)} \cap [0,1)^{2d}} | Z \varphi (x + u , \omega + \eta) |^2$, $\ell = 0, 1, \ldots , N-1$ is nonzero.
\end{itemize}
Moreover, if any one of the above holds, $\mathcal{G}(\varphi, \Lambda)$ is the orthogonal direct sum
$$\mathcal{G}(\varphi, \Lambda) = U^{(0)} \oplus \cdots \oplus U^{(N-1)}$$
with each $U^{(\ell)}$ being a (possibly trivial) subspace of $\mathcal{G}(\varphi, \Lambda)$ which is $\widetilde{\Lambda}$-invariant.
\end{theorem}

\begin{remark}
\label{rmk_after_thm_Gabor_integer_lattice}
\rm
(a) By scaling the Zak transform as
$
Z_{\alpha} f ( x , \omega ) = \sum_{k \in \mathbb{Z}^d} f (x + \alpha k) \, e^{-2 \pi i \alpha k \cdot \omega}
$ where $\alpha > 0$,
Theorem \ref{thm_Gabor_integer_lattice} can be generalized to the case where $\Lambda \subseteq \widetilde{\Lambda} \subseteq \alpha \mathbb{Z}^d \times \tfrac{1}{\alpha} \mathbb{Z}^d$, $\alpha > 0$. \\
(b) It is easily seen that each of $I^{(\ell)}$, $\ell = 0, 1, \ldots , N-1$ is of the form $\{ (u,\eta) \in \mathbb{R}^{2d} \, : \, e^{2 \pi i (b \cdot u - a \cdot \eta)} = \zeta_N \}$ where $\zeta_N$ is an $N$ th root of unity.
While proving Theorem \ref{thm_Gabor_integer_lattice}, we will assume without loss of generality that
\begin{align*}
I^{(\ell)}
= \{ (u,\eta) \in \mathbb{R}^{2d} \, : \, e^{2 \pi i (b \cdot u - a \cdot \eta)} = e^{2 \pi i \ell / N}
\quad \text{for all} \; (a,b) \in \widetilde{\Lambda} \} ,
\quad \ell = 0, 1, \ldots , N-1 .
\end{align*}
(c) Let $\mathcal{K} \subset {\Lambda}^{\circ}$ be a set of representatives of the quotient ${\Lambda}^{\circ} / \widetilde{\Lambda}^{\circ} = \{ I^{(0)} , I^{(1)}, \ldots, I^{(N-1)} \}$, so that $\mathcal{K}$ consists of exactly $N$ elements each of which represents one $I^{(\ell)}$. If $D \subset [0,1)^{2d}$ is a fundamental domain of the lattice ${\Lambda}^{\circ}$, then the finite union $\widetilde{D} = \bigcup_{(u,\eta) \in \mathcal{K}} (u,\eta) + D$ is a fundamental domain of the coarser lattice $\widetilde{\Lambda}^{\circ}$.
The $\widetilde{\Lambda}^{\circ}$-periodization of $D$ is the set $B^{(0)}$, while the $\widetilde{\Lambda}^{\circ}$-periodization $\widetilde{D}$ is $\mathbb{R}^2$.
\end{remark}

For the proof of Theorem \ref{thm_Gabor_integer_lattice}, we need the following lemma.

\begin{lemma}[cf.~Lemma 4.3 in \cite{ACHKM10}]
\label{lem_U_ell_inclusion}
Under the same assumptions as in Theorem \ref{thm_Gabor_integer_lattice}, if $U^{(\ell)} \subseteq \mathcal{G}(\varphi, \Lambda)$ for some $\ell$, then it is a $\widetilde{\Lambda}$-invariant closed subspace of $\mathcal{G}(\varphi, \Lambda)$.
\end{lemma}

\begin{proof}
The proof is similar to Lemma 4.3 in \cite{ACHKM10}.

Assume that $U^{(\ell)} \subseteq \mathcal{G}(\varphi, \Lambda)$ for some $\ell$.
To see that $U^{(\ell)}$ is closed, suppose that $\{ f_n \}_{n = 1}^{\infty} \subset U^{(\ell)}$ is a sequence that converges to some $f$ in $L^2(\mathbb{R}^d)$. Since $\mathcal{G}(\varphi, \Lambda)$ is closed and $\{ f_n \}_{n = 1}^{\infty} \subset \mathcal{G}(\varphi, \Lambda)$, it follows that $f \in \mathcal{G}(\varphi, \Lambda)$.
Further, since $Z$ is unitary we have
\begin{align*}
\| f_n - f \|^2_{L^2(\mathbb{R}^d)}
&= \| Z (f_n - f) \|^2_{L^2([0,1]^{2d})} \\
&= \| (Z f_n - Z f) \chi_{B^{(\ell)}} \|^2_{L^2([0,1]^{2d})} +  \| (Z f_n - Z f) \chi_{{B^{(\ell)}}^C} \|^2_{L^2([0,1]^{2d})} \\
&= \| Z f_n - Z f \cdot \chi_{B^{(\ell)}} \|^2_{L^2([0,1]^{2d})} +  \| Z f \cdot \chi_{{B^{(\ell)}}^C} \|^2_{L^2([0,1]^{2d})}.
\end{align*}
Since the left hand side converges to zero, we must have $Z f_n \rightarrow Z f \cdot \chi_{B^{(\ell)}}$ in $L^2([0,1]^{2d})$ and $Z f \cdot \chi_{{B^{(\ell)}}^C} = 0$. Since $Z f_n \rightarrow Z f$ in $L^2([0,1]^{2d})$, we have $Z f  = Z f \cdot \chi_{B^{(\ell)}}$ which together with $f \in \mathcal{G}(\varphi, \Lambda)$ implies that $f \in U^{(\ell)}$. Thus, $U^{(\ell)}$ is a closed subspace of $\mathcal{G}(\varphi, \Lambda)$.

Let us first see that $U^{(\ell)}$ is $\Lambda$-invariant. Fix any $f \in U^{(\ell)}$ and let $g \in \mathcal{G}(\varphi, \Lambda)$ be such that $Z f = Z g \cdot \chi_{B^{(\ell)}}$.
For any $(a,b) \in \Lambda \; (\subseteq \mathbb{Z}^{2d})$, we have
\begin{align*}
\Big( Z \pi (a, b) f \Big) (x, \omega)
&= e^{i 2 \pi (b \cdot x - a \cdot \omega)} \, Z f ( x , \omega )
= e^{i 2 \pi (b \cdot x - a \cdot \omega)} \, Z g ( x , \omega ) \cdot \chi_{B^{(\ell)}} ( x , \omega ) \\
&= \Big( Z \pi (a, b) g \Big) (x, \omega) \cdot \chi_{B^{(\ell)}} ( x , \omega ) ,
\end{align*}
where $\pi (a, b) g \in \mathcal{G}(\varphi, \Lambda)$, and thus $\pi (a, b) f \in U^{(\ell)}$. This shows that $U^{(\ell)}$ is $\Lambda$-invariant.

Next, to see that $U^{(\ell)}$ is in fact $\widetilde{\Lambda}$-invariant, we fix any $(a,b) \in \widetilde{\Lambda} \; (\subseteq \mathbb{Z}^{2d})$ and consider a ${\Lambda}^{\circ}$-periodic function given by
\begin{align*}
h^{(\ell)}_{a,b} (x,\omega)
&= \tfrac{1}{M} \, e^{- 2 \pi i \ell / N} \sum_{(u,\eta) \in {\Lambda}^{\circ} \cap [0,1)^{2d}} e^{2 \pi i [b \cdot (x+u) - a \cdot (\omega+\eta)] } \, \chi_{B^{(\ell)}} (x + u , \omega + \eta) ,
\end{align*}
where $M = |\widetilde{\Lambda}^{\circ} / \mathbb{Z}^{2d}| = |\widetilde{\Lambda}^{\circ} \cap [0,1)^{2d}|$.
Then
\begin{align*}
h^{(\ell)}_{a,b} (x,\omega) \, \chi_{B^{(\ell)}} (x , \omega)
&=\tfrac{1}{M} \, e^{- 2 \pi i \ell / N} \sum_{(u,\eta) \in {\Lambda}^{\circ} \cap [0,1)^{2d}} e^{2 \pi i [b \cdot (x+u) - a \cdot (\omega+\eta)] } \, \chi_{B^{(\ell)}} (x + u , \omega + \eta) \, \chi_{B^{(\ell)}} (x , \omega) \\
&=
\tfrac{1}{M} \, e^{- 2 \pi i \ell / N} \sum_{(u,\eta) \in I^{(\ell)} \cap [0,1)^{2d} } e^{2 \pi i \ell / N} \, e^{2 \pi i (b \cdot x - a \cdot \omega) } \, \chi_{B^{(\ell)}} (x + u , \omega + \eta) \\
&= e^{2 \pi i (b \cdot x - a \cdot \omega) } \, \chi_{B^{(\ell)}} (x , \omega ) .
\end{align*}
For any $f \in U^{(\ell)}$, since $\supp Zf \subseteq B^{(\ell)}$, we have
\begin{align*}
\Big( Z \pi (a,b) f \Big) (x, \omega)
&= e^{2 \pi i (b \cdot x - a \cdot \omega) } \, Z f ( x , \omega ) = h^{(\ell)}_{a,b} (x,\omega) \, Z f ( x , \omega ) .
\end{align*}
By Lemma \ref{lem_Gabor_ftn_exp_lambda_integer} and since $U^{(\ell)}$ is $\Lambda$-invariant, it follows that $\pi (a,b) f \in \mathcal{G}(f, \Lambda) \subseteq U^{(\ell)}$.
Therefore, $U^{(\ell)}$ is $\widetilde{\Lambda}$-invariant.
\end{proof}

\medskip

\noindent \textit{Proof of Theorem \ref{thm_Gabor_integer_lattice}.}\;
(a) $\Rightarrow$ (b):
Assume that $\mathcal{G}(\varphi, \Lambda) = \mathcal{G}(\varphi, \widetilde{\Lambda})$.
Fix any $\ell = 0, 1, \ldots , N-1$.
If $f \in U^{(\ell)}$, then exists $g \in \mathcal{G}(\varphi, \Lambda)$ such that $Z f  = Z g \cdot \chi_{B^{(\ell)}}$.
Since $B^{(\ell)}$ is periodic with respect to $\widetilde{\Lambda}^{\circ}$, it follows by Lemma \ref{lem_Gabor_ftn_exp_lambda_integer} that $f \in \mathcal{G}(g, \widetilde{\Lambda})$. Since $\mathcal{G}(g, \widetilde{\Lambda}) \subseteq \mathcal{G}(\varphi, \widetilde{\Lambda}) = \mathcal{G}(\varphi, \Lambda)$, we conclude that $U^{(\ell)} \subseteq \mathcal{G}(\varphi, \Lambda)$. \\
(b) $\Rightarrow$ (a): Assume that $U^{(\ell)} \subseteq \mathcal{G}(\varphi, \Lambda)$ for all $\ell = 0, 1, \ldots , N-1$. Then Lemma \ref{lem_U_ell_inclusion} implies that all $U^{(\ell)}$, $\ell = 0, 1, \ldots , N-1$ are $\widetilde{\Lambda}$-invariant closed subspaces of $\mathcal{G}(\varphi, \Lambda)$.
These subspaces are mutually orthogonal, since the sets $B^{(\ell)}$, $\ell = 0, 1, \ldots , N-1$ are disjoint.
Moreover, every $f \in \mathcal{G}(\varphi, \Lambda)$ can be decomposed as $f = f^{(0)} + \ldots + f^{(N-1)}$, where $f^{(\ell)} = Z^{-1} (Z f \cdot \chi_{B^{(\ell)}}) \in U^{(\ell)}$ for $\ell = 0, 1, \ldots , N-1$.
Therefore, we have the orthogonal direct sum decomposition
$$\mathcal{G}(\varphi, \Lambda) = U^{(0)} \oplus \cdots \oplus U^{(N-1)}. $$
Since all $U^{(\ell)}$ are $\widetilde{\Lambda}$-invariant, so is $\mathcal{G}(\varphi, \Lambda)$. \\
(b) $\Rightarrow$ (c): This is trivial, since $\varphi \in \mathcal{G}(\varphi, \Lambda)$. \\
(c) $\Rightarrow$ (d): Assume that $Z^{-1} ( Z \varphi \cdot \chi_{B^{(\ell)}} ) \in \mathcal{G}(\varphi, \Lambda)$ for all $\ell = 0, 1, \ldots , N-1$.
Then for each $\ell$, Lemma \ref{lem_Gabor_ftn_exp_lambda_integer} implies that there exists a ${\Lambda}^{\circ}$-periodic measurable function $h^{(\ell)} ( x , \omega )$ such that
\begin{align*}
Z \varphi ( x , \omega ) \cdot \chi_{B^{(\ell)}} ( x , \omega ) = h^{(\ell)} ( x , \omega ) \,  Z \varphi ( x , \omega ) .
\end{align*}
By a standard periodization trick, we get
\begin{align}
\label{eqn_peri_trick1}
\sum_{(u,\eta) \in {\Lambda}^{\circ} \cap [0,1)^{2d}} | Z \varphi (x + u , \omega + \eta) |^2 \, \chi_{B^{(\ell)}} (x + u , \omega + \eta)
= | h^{(\ell)} ( x , \omega ) |^2 \sum_{(u,\eta) \in {\Lambda}^{\circ} \cap [0,1)^{2d}} | Z \varphi (x + u , \omega + \eta) |^2  ,
\end{align}
and
\begin{align}
\label{eqn_peri_trick2}
\chi_{B^{(\ell)}} ( x , \omega )  \sum_{(u,\eta) \in \widetilde{\Lambda}^{\circ} \cap [0,1)^{2d}} | Z \varphi (x + u , \omega + \eta) |^2
= | h^{(\ell)} ( x , \omega ) |^2 \sum_{(u,\eta) \in \widetilde{\Lambda}^{\circ} \cap [0,1)^{2d}} | Z \varphi (x + u , \omega + \eta) |^2  .
\end{align}
Note that the left hand sides of (\ref{eqn_peri_trick1}) and (\ref{eqn_peri_trick2}) coincide if $( x , \omega ) \in B^{(\ell)}$. Thus, for a.e.~$(x,\omega) \in B^{(\ell)}$,
\begin{align*}
\sum_{(u,\eta) \in \widetilde{\Lambda}^{\circ} \cap [0,1)^{2d}} | Z \varphi (x + u , \omega + \eta) |^2
&= | h^{(\ell)} ( x , \omega ) |^2 \sum_{(u,\eta) \in \widetilde{\Lambda}^{\circ} \cap [0,1)^{2d}} | Z \varphi (x + u , \omega + \eta) |^2 \\
&= | h^{(\ell)} ( x , \omega ) |^2 \sum_{(u,\eta) \in {\Lambda}^{\circ} \cap [0,1)^{2d}} | Z \varphi (x + u , \omega + \eta) |^2 ,
\end{align*}
from which we see that if $\sum_{(u,\eta) \in \widetilde{\Lambda}^{\circ} \cap [0,1)^{2d}} | Z \varphi (x + u , \omega + \eta) |^2 \neq 0$, then $| h^{(\ell)} ( x , \omega ) |^2 = 1$ and in turn, $\sum_{(u,\eta) \in ({\Lambda}^{\circ} \backslash \widetilde{\Lambda}^{\circ}) \cap [0,1)^{2d}} | Z \varphi (x + u , \omega + \eta) |^2 = 0$.
Since the sets $B^{(\ell)}$, $\ell = 0, 1, \ldots , N-1$ form a partition of $\mathbb{R}^{2d}$, we conclude that for a.e.~$(x,\omega) \in \mathbb{R}^{2d}$, $\sum_{(u,\eta) \in \widetilde{\Lambda}^{\circ} \cap [0,1)^{2d}} | Z \varphi (x + u , \omega + \eta) |^2 \neq 0$ implies $\sum_{(u,\eta) \in ({\Lambda}^{\circ} \backslash \widetilde{\Lambda}^{\circ}) \cap [0,1)^{2d}} | Z \varphi (x + u , \omega + \eta) |^2 = 0$.
Then (d) follows by observing that Zak transform is quasi-periodic. \\
(d) $\Rightarrow$ (a):
Assume that (d) holds, and fix any $(a,b) \in \widetilde{\Lambda} \; (\subseteq \mathbb{Z}^{2d})$.
We will show $\pi(a,b) \varphi \in \mathcal{G}(\varphi, \Lambda)$ using Lemma \ref{lem_Gabor_ftn_exp_lambda_integer}, more precisely, by constructing a ${\Lambda}^{\circ}$-periodic measurable function $h: \mathbb{R}^{2d} \rightarrow \mathbb{C}$ such that $(Z \pi(a,b) \varphi) ( x , \omega ) = h ( x , \omega ) \,  Z \varphi ( x , \omega )$.
Noting that $D$ is a fundamental domain of the lattice ${\Lambda}^{\circ}$, we will define $h$ on $D$ and extend it ${\Lambda}^{\circ}$-periodically to $\mathbb{R}^{2d}$.
By assumption, the set of all $(x , \omega) \in D$ for which (\ref{eqn_cond_d_of_thm_Gabor_integer_lattice}) is violated is a measure zero set which we denote by $D_0 \; (\subset D)$.
Define $h (x, \omega) = 0$ for $(x,\omega) \in D_0$.
Next, fix any $( x , \omega )$ in $D \backslash D_0$.
\begin{itemize}
\item
If $\sum_{(u,\eta) \in I^{(\ell)} \cap [0,1)^{2d}} | Z \varphi (x + u , \omega + \eta) |^2 = 0$ for all $\ell = 0, 1, \ldots , N-1$, equivalently, if $Z \varphi ( x + u , \omega + \eta ) = 0$ for all $(u,\eta) \in {\Lambda}^{\circ} \cap [0,1)^{2d}$, then define $h (x, \omega) = 0$.

\item
Otherwise, there exists a unique $0 \leq \ell_0 \leq N-1$ such that $\sum_{(u,\eta) \in I^{(\ell)} \cap [0,1)^{2d}} | Z \varphi (x + u , \omega + \eta) |^2 = 0$ for all $\ell$ except $\ell_0$, equivalently, $Z \varphi ( x + u , \omega + \eta ) = 0$ for all $(u,\eta) \in ({\Lambda}^{\circ} \backslash I^{(\ell_0)}) \cap [0,1)^{2d}$.
We define $h (x, \omega) = e^{2 \pi i \ell_0 / N} \cdot e^{2 \pi i (b \cdot x - a \cdot \omega) }$.
Observe that for any $(u,\eta) \in I^{(\ell_0)} \cap [0,1)^{2d}$, we have $e^{2 \pi i [b \cdot (x + u) - a \cdot (\omega + \eta)] } = e^{2 \pi i \ell_0 / N} \cdot e^{2 \pi i (b \cdot x - a \cdot \omega) } = h (x, \omega)$.
Combining with the fact that $Z \varphi ( x + u , \omega + \eta ) = 0$ for $(u,\eta) \in ({\Lambda}^{\circ} \backslash I^{(\ell_0)}) \cap [0,1)^{2d}$, we obtain that for all $(u,\eta) \in {\Lambda}^{\circ} \cap [0,1)^{2d}$,
$$
e^{2 \pi i [b \cdot (x + u) - a \cdot (\omega + \eta)] } \, Z \varphi ( x + u , \omega + \eta ) = h (x, \omega) \, Z \varphi ( x + u , \omega + \eta ) .
$$
\end{itemize}
With $h (x, \omega)$ defined on $D$ as above, it follows that for all $(u,\eta) \in {\Lambda}^{\circ} \cap [0,1)^{2d}$,
$$
e^{2 \pi i [b \cdot (x + u) - a \cdot (\omega + \eta)] } \, Z \varphi ( x + u , \omega + \eta ) = h (x, \omega) \, Z \varphi ( x + u , \omega + \eta ) ,
\quad \text{a.e.} \; ( x , \omega ) \in D .
$$
This in fact holds for all $(u,\eta) \in {\Lambda}^{\circ}$, since ${\Lambda}^{\circ} \supseteq \mathbb{Z}^{2d}$ and Zak transform is quasi-periodic.
Therefore, with $h (x, \omega)$ extended ${\Lambda}^{\circ}$-periodically from $D$ to $\mathbb{R}^{2d}$, we have
$$
e^{2 \pi i (b \cdot x - a \cdot \omega) } \, Z \varphi ( x , \omega )
= h (x, \omega) \, Z \varphi ( x , \omega )
\quad \text{a.e.},
$$
From (\ref{Zak_on_pi_quasi_peri}) and Lemma \ref{lem_Gabor_ftn_exp_lambda_integer}, we conclude that $\pi(a,b) \varphi \in \mathcal{G}(\varphi, \Lambda)$.
   \hfill $\Box$ %\hfill $\square$ %end of proof

\begin{corollary}
\label{cor_Gabor_integer_lattice}
Let $\varphi \in L^2(\mathbb{R}^d)$ and let $\Lambda \subseteq \mathbb{Z}^{2d}$ be a lattice.
Then $\mathcal{G}(\varphi, \Lambda)$ is $\mathbb{Z}^{2d}$-invariant if and only if $Z \varphi (x , \omega )$ vanishes a.e.~on $[0,1)^{2d} \backslash D$, where $D \subset [0,1)^{2d}$ is a fundamental domain  of the lattice ${\Lambda}^{\circ}$.
When $(\varphi, \Lambda)$ is a Riesz basis for $\mathcal{G}(\varphi, \Lambda)$, the latter condition is refined to: $Z \varphi (x , \omega ) \neq 0$ a.e.~on $D$ and $Z \varphi (x , \omega ) = 0$ a.e.~on $[0,1)^{2d} \backslash D$.
\end{corollary}

\begin{proof}
By Theorem \ref{thm_Gabor_integer_lattice} and the quasi-periodicity of Zak transform, it follows that $\mathcal{G}(\varphi, \Lambda)$ is $\mathbb{Z}^{2d}$-invariant if and only if for a.e.~$(x,\omega)$, $Z \varphi (x , \omega ) \neq 0$ implies that $Z \varphi ( x + u , \omega + \eta ) = 0$ for all $(u,\eta) \in ({\Lambda}^{\circ} \cap [0,1)^{2d}) \backslash \{ (0,0) \}$.
The latter is equivalent to that for a.e.~$(x,\omega)$, we have $Z \varphi ( x + u , \omega + \eta ) \neq 0$ for at most one $(u,\eta)$ in ${\Lambda}^{\circ} \cap [0,1)^{2d}$, which holds if and only if $Z \varphi (x , \omega )$ vanishes a.e.~on $[0,1)^{2d} \backslash D$ where $D \subset [0,1)^{2d}$ is a fundamental domain of the lattice ${\Lambda}^{\circ}$.

For the second part, observe that $(\varphi, \Lambda)$ is a Riesz basis for $\mathcal{G}(\varphi, \Lambda)$ with Riesz bounds $B \geq A > 0$ if and only if
\begin{align*}
m A \leq \sum_{(u,\eta) \in {\Lambda}^{\circ} \cap [0,1)^{2d}} | Z \varphi (x + u , \omega + \eta) |^2 \leq m B
\quad \text{a.e.},
\end{align*}
where $m = |{\Lambda}^{\circ} \cap [0,1)^{2d}| \geq 1$.
In this case, we have $Z \varphi (x , \omega ) \neq 0$ a.e.~at least on a fundamental domain of the lattice ${\Lambda}^{\circ}$. The claim is then straightforward.
\end{proof}

\begin{remark}
\label{rmk_Gabor_extra_inv}
\rm
In Proposition \ref{prop_SIS_gen_inv} extra invariance of $\mathcal{S} (\varphi, \mathbb{Z}^{2d})$ is characterized through zeros of $\widehat{\varphi}(\xi,\omega)$, while in Theorem \ref{thm_Gabor_integer_lattice} extra invariance of $\mathcal{G}(\varphi, \Lambda)$ is characterized through zeros of $Z \varphi (x , \omega )$.
Similar to Remark \ref{rmk_SIS_extra_inv}, we have following.

Assume that $\mathcal{G}(\varphi, \Lambda) = \mathcal{G}(\varphi, \widetilde{\Lambda})$, where $\varphi \in L^2(\mathbb{R}^{2d})$ and $\Lambda, \widetilde{\Lambda} \subseteq \mathbb{R}^{2d}$ are lattices satisfying $\Lambda \subsetneq \widetilde{\Lambda} \subseteq \mathbb{Z}^{2d}$ (so that ${\Lambda}^{\circ} \supsetneq \widetilde{\Lambda}^{\circ} \supseteq \mathbb{Z}^{2d}$).
Then every $f \in \mathcal{G}(\varphi, \Lambda)$ can be represented in two different ways (in the Zak transform domain):
$$
h ( x , \omega ) \,  Z \varphi ( x , \omega ) = Z f ( x , \omega ) = \widetilde{h} ( x , \omega ) \,  Z \varphi ( x , \omega ),
$$
where $h ( x , \omega )$ is ${\Lambda}^{\circ}$-periodic and $\widetilde{h} ( x , \omega )$ is $\widetilde{\Lambda}^{\circ}$-periodic (see Lemma \ref{lem_Gabor_ftn_exp_lambda_integer}), and thus we have
\begin{align*}
h ( x , \omega ) = \widetilde{h} ( x , \omega )
\quad \text{for a.e.} \; ( x , \omega ) \;\; \text{such that} \; Z \varphi ( x , \omega ) \neq 0.
\end{align*}
By picking $h ( x , \omega )$ and $\widetilde{h} ( x , \omega )$ that are genuinely ${\Lambda}^{\circ}$-periodic and $\widetilde{\Lambda}^{\circ}$-periodic respectively, and exploiting the fact that $\widetilde{\Lambda}^{\circ} \subsetneq {\Lambda}^{\circ}$, we obtain some restrictions on the set $\{ (x,\omega) \, : \, Z \varphi ( x , \omega ) \neq 0 \}$ which is defined up to a measure zero set.
Clearly, it is impossible that $Z \varphi ( x , \omega ) \neq 0$ for a.e.~$( x , \omega )$ in $\mathbb{R}^{2d}$. This yields condition (d) in Theorem \ref{thm_Gabor_integer_lattice}.
\end{remark}

\medskip

\subsubsection{The case $\Lambda = \mathbb{Z}^{2d}$ with $\widetilde{\Lambda} = \mathbb{R}^d \times \mathbb{Z}^d$, $\mathbb{Z}^d \times \mathbb{R}^d$, $\mathbb{R}^{2d}$.}
\label{subsubsec:ZxZ_case}
$ $

To compare with shift-invariant spaces $\mathcal{S}(\varphi, \mathbb{Z}^{2d})$, we consider the lattice $\Lambda = \mathbb{Z}^{2d}$.
We will treat the extreme cases $\widetilde{\Lambda} = \mathbb{R}^d \times \mathbb{Z}^d$, $\mathbb{Z}^d \times \mathbb{R}^d$, $\mathbb{R}^{2d}$.
Note that since $\mathbb{R}^d \times \mathbb{Z}^d$ is the smallest closed subgroup of $\mathbb{R}^{2d}$ containing $\mathbb{R}^d \times \{ 0 \}$ and $\mathbb{Z}^{2d}$, the space $\mathcal{G}(\varphi, \mathbb{Z}^{2d})$ is $\mathbb{R}^d \times \mathbb{Z}^d$-invariant if and only if it is $\mathbb{R}^d \times \{ 0 \}$-invariant.
Similarly, the space $\mathcal{G}(\varphi, \mathbb{Z}^{2d})$ is $\mathbb{Z}^d \times \mathbb{R}^d$-invariant if and only if it is $\{ 0 \} \times \mathbb{R}^d$-invariant.

\begin{proposition}
\label{prop_Gabor_ZxZ trans_inv_mod_inv}
Let $\varphi \in L^2(\mathbb{R}^d)$.
\begin{itemize}
\item[(a)]
$\mathcal{G}(\varphi, \mathbb{Z}^{2d})$ is invariant under all translations ($\mathbb{R}^d \times \{ 0 \}$-invariant) if and only if there exists a measurable set $E \subset [0,1)^d$ such that $Z \varphi (x , \omega ) \neq 0$ a.e.~on $[0,1)^d \times E$ and $Z \varphi (x , \omega ) = 0$ a.e.~on $[0,1)^d \times ([0,1)^d \backslash E)$.

\item[(b)]
$\mathcal{G}(\varphi, \mathbb{Z}^{2d})$ is invariant under all modulations ($\{ 0 \} \times \mathbb{R}^d$-invariant) if and only if there exists a measurable set $E \subset [0,1)^d$ such that $Z \varphi (x , \omega ) \neq 0$ a.e.~on $E \times [0,1)^d$ and $Z \varphi (x , \omega ) = 0$ a.e.~on $([0,1)^d \backslash E) \times [0,1)^d$.

\item[(c)]
$\mathcal{G}(\varphi, \mathbb{Z}^{2d})$ is invariant under all time-frequency shifts ($\mathbb{R}^{2d}$-invariant) if and only if $\mathcal{G}(\varphi, \mathbb{Z}^{2d})$ is either $\{ 0 \}$ or $L^2(\mathbb{R}^d)$.
Consequently, a nontrivial proper Gabor subspace $\mathcal{G}(\varphi, \mathbb{Z}^{2d})$ of $L^2(\mathbb{R}^d)$ cannot be invariant under all time-frequency shifts.
\end{itemize}
\end{proposition}

\begin{proof}
For any $(u,\eta) \in \mathbb{R}^d \times \mathbb{R}^d$, Lemma \ref{lem_Gabor_ftn_exp_lambda_integer} together with (\ref{Zak_on_pi}) implies that $\pi(u,\eta) \varphi \in \mathcal{G}(\varphi, \mathbb{Z}^{2d})$ if and only if there exists $\mathbb{Z}^{2d}$-periodic measurable function $h ( x , \omega )$ satisfying $e^{2 \pi i \eta \cdot x } \, Z \varphi ( x - u , \omega - \eta) = h ( x , \omega ) \, Z \varphi ( x , \omega )$ for a.e.~$( x , \omega ) \in [0,1)^{2d}$.\\
(a) ($\Leftarrow$) Assume that $E \subset [0,1)^d$ is a measurable set such that $Z \varphi (x , \omega ) \neq 0$ a.e.~on $[0,1)^d \times E$ and $Z \varphi (x , \omega ) = 0$ a.e.~on $[0,1)^d \times ([0,1)^d \backslash E)$, and fix any $u \in \mathbb{R}^d$. For a.e.~$\omega_0 \in E$ fixed, we have $Z \varphi ( \cdot , \omega_0 ) \neq 0$ a.e.~and thus exists a measurable function $h ( \cdot , \omega_0 )$ such that $Z \varphi ( \cdot - u , \omega_0) = h ( \cdot , \omega_0 ) \, Z \varphi ( \cdot , \omega_0 )$ a.e.
For a.e.~$\omega_0 \in [0,1)^d \backslash E$ fixed, we always have $Z \varphi ( \cdot - u , \omega_0) = 0 = h ( \cdot , \omega_0 ) \, Z \varphi ( \cdot , \omega_0 )$ a.e.~so we may set $h ( \cdot , \omega_0 ) = 0$.
With $h ( x , \omega )$ defined on $[0,1)^{2d}$ as above (and extended $\mathbb{Z}^{2d}$-periodically over $\mathbb{R}^d$), we have $Z \varphi ( x - u , \omega) = h ( x , \omega ) \, Z \varphi ( x , \omega )$ for a.e.~$( x , \omega ) \in [0,1)^d \times [0,1)^d$. \\
($\Rightarrow$)
Suppose to the contrary that $S \subset [0,1)^d  \times E_1$ is a measurable set with $0 < \mu(S) < \mu (E_1)$ such that $Z \varphi (x , \omega ) \neq 0$ a.e.~on $S$ and $Z \varphi (x , \omega ) = 0$ a.e.~on $([0,1)^d  \times E_1) \backslash S$, where $E_1 \subset [0,1)^d$ is a set of positive measure satisfying $0 < \mu(\{ x \in [0,1)^d \, : \, Z \varphi (x , \omega_0 ) \neq 0 \}) < 1$ for every $\omega_0 \in E_1$. Here $\mu(\cdot)$ denotes the Lebesgue measure.
Then there exist $u \in \mathbb{R}^d$ and an open set $U \subset [0,1)^d  \times E_1$ such that $\mu (U \cap [S + (u,0)] ) \geq \tfrac{3}{4} \mu(U)$ and $\mu ( U \cap [([0,1)^d  \times E_1) \backslash S] ) \geq \tfrac{3}{4} \mu(U)$.
Note that the sets $S + (u,0)$ and $([0,1)^d  \times E_1) \backslash S$ intersect on a set of Lebesgue measure at least $\mu(U) / 2$ which we will denote by $W$.
Since $\pi(u,0) \varphi \in \mathcal{G}(\varphi, \mathbb{Z}^{2d})$, there exists a $\mathbb{Z}^{2d}$-periodic measurable function $h ( x , \omega )$ satisfying $Z \varphi ( x - u , \omega ) = h ( x , \omega ) \, Z \varphi ( x , \omega )$ for a.e.~$( x , \omega ) \in [0,1)^{2d}$. However, for a.e.~$( x , \omega ) \in W$, we have $0 \neq Z \varphi ( x - u , \omega ) = h ( x , \omega ) \, Z \varphi ( x , \omega ) = 0$, which is a contradiction. \\
(b) The proof of (b) is similar to (a). \\
(c) The implication ($\Leftarrow$) is obvious. \\
($\Rightarrow$) Assume that $\mathcal{G}(\varphi, \mathbb{Z}^{2d})$ is invariant under all time-frequency shifts. From (a) and (b), it follows that either (i) $Z \varphi (x , \omega ) = 0$ a.e.~on $[0,1)^{2d}$ or (ii) $Z \varphi (x , \omega ) \neq 0$ a.e.~on $[0,1)^{2d}$. The proof is complete by observing that each (i) and (ii) corresponds to $\mathcal{G}(\varphi, \mathbb{Z}^{2d}) = \{ 0 \}$ and $\mathcal{G}(\varphi, \mathbb{Z}^{2d}) = L^2(\mathbb{R}^d)$, respectively (cf.~\cite[Theorem 4.3.3]{HW89}).
\end{proof}

\begin{remark}[Shift-invariant spaces vs.~Gabor spaces generated by integer lattices]
\label{rmk_SIS_vs_Gabor}
\rm
$ $\\
(i) There is no $\varphi \in L^2(\mathbb{R}^d)$ such that $\mathcal{S} (\varphi, \mathbb{Z}^d) = L^2(\mathbb{R}^d)$.
Indeed, with $\varphi \in L^2(\mathbb{R}^d)$ fixed, not every function $\widehat{f}(\xi)$ of $L^2(\mathbb{R}^d)$ can be expressed in the form $\widehat{f}(\xi) = m(\xi) \, \widehat{\varphi}(\xi)$ where $m (\xi)$ is $\mathbb{Z}^d$-periodic, hence $\mathcal{S} (\varphi, \mathbb{Z}^d) \neq L^2(\mathbb{R}^d)$. \\
(ii) There exists $\varphi \in L^2(\mathbb{R}^d)$ such that $\mathcal{G}(\varphi, \mathbb{Z}^{2d}) = L^2(\mathbb{R}^d)$.
For example, $\mathcal{G}(\chi_{[0,1)^d} , \mathbb{Z}^{2d}) = L^2(\mathbb{R}^d)$ where $\chi_{[0,1)^d}$ is the characteristic function of $[0,1)^d$.
In fact, since $| Z \chi_{[0,1)^d} ( x , \omega )| = 1$ for all $(x,\omega) \in \mathbb{R}^d \times \mathbb{R}^d$, the Gabor system $(\chi_{[0,1)^d} , \mathbb{Z}^{2d})$ is an orthonormal basis for $L^2(\mathbb{R}^d)$ (cf.~\cite[Corollary 8.3.2]{Gro01}). \\
(iii) There exists a nontrivial proper shift-invariant space $\mathcal{S} (\varphi, \mathbb{Z}^d)$ of $L^2(\mathbb{R}^d)$ which is invariant under all translations
 (cf.~Proposition \ref{prop_SIS_trans_inv}).
For example, the shift-invariant space generated by $\varphi(x)=\sin(\pi x)/ (\pi x)$ is invariant under all translations. This space is also known as the Paley-Wiener space of signals bandlimited to $[-1/2, 1/2]$. \\
(iv) There is no nontrivial proper Gabor subspace $\mathcal{G}(\varphi, \mathbb{Z}^{2d})$ of $L^2(\mathbb{R}^d)$ which is invariant under all time-frequency shifts (Proposition \ref{prop_Gabor_ZxZ trans_inv_mod_inv}).
\end{remark}

\medskip

\begin{example}
\label{example_4Zx2Z}
\rm
We consider a Gabor space $\mathcal{G}(\varphi, 4 \mathbb{Z} \times 2 \mathbb{Z})$ where $\varphi \in L^2(\mathbb{R})$, which corresponds to the case ($d = 1$, $p_1 = 4$, $p_2 = 2$).

First, we pick a pair $(a,b)$ in $\mathbb{Z}_{4} \times \mathbb{Z}_{2}$ and let $\widetilde{\Lambda} \subset \mathbb{R}^2$ be the smallest closed subgroup of $\mathbb{R}^2$ containing $4 \mathbb{Z} \times 2 \mathbb{Z}$ and $(a,b)$.
Then $4 \mathbb{Z} \times 2 \mathbb{Z} \subseteq \widetilde{\Lambda} \subseteq \mathbb{Z}^2$ so that $\mathbb{Z}^2 \subseteq \widetilde{\Lambda}^{\circ} \subseteq \tfrac{1}{2} \mathbb{Z} \times \tfrac{1}{4} \mathbb{Z}$.
For illustration of $\widetilde{\Lambda} \supseteq 4 \mathbb{Z} \times 2 \mathbb{Z}$, we observe $\widetilde{\Lambda}$ in the region $[0,4) \times [0,2)$ which is a fundamental domain of $4 \mathbb{Z} \times 2 \mathbb{Z}$; the generating element $(a,b) \in \widetilde{\Lambda}$ is marked in red.
Likewise, for illustration of $\widetilde{\Lambda}^{\circ} \supseteq \mathbb{Z}^2$, we observe $\widetilde{\Lambda}^{\circ}$ in the region $[0,1)^2$ which is a fundamental domain of $\mathbb{Z}^2$; the complement of $\widetilde{\Lambda}^{\circ}$ in $\tfrac{1}{2} \mathbb{Z} \times \tfrac{1}{4} \mathbb{Z}$, i.e., $(\tfrac{1}{2} \mathbb{Z} \times \tfrac{1}{4} \mathbb{Z}) \backslash \widetilde{\Lambda}^{\circ}$, is marked as empty nodes.
Overlapped with $\widetilde{\Lambda}^{\circ}$, we depict the set $B^{(0)}$ which is the $\widetilde{\Lambda}^{\circ}$-periodic extension of $[0,\tfrac{1}{2}) \times [0,\tfrac{1}{4})$ to $\mathbb{R}^{2}$.
In all figures, the edges of $[0,1)^2$ are drawn in thick line.

\medskip

\noindent
\begin{minipage}{0.5\columnwidth}
\noindent
(i) If $(a,b) = (0,0)$, then \\
$\widetilde{\Lambda} = 4 \mathbb{Z} \times 2 \mathbb{Z}$,
$\widetilde{\Lambda}^{\circ} = \tfrac{1}{2} \mathbb{Z} \times \tfrac{1}{4} \mathbb{Z}$.
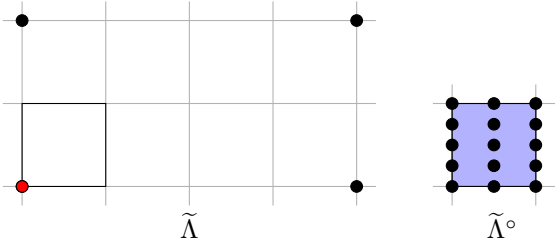
\begin{center}
\begin{tikzpicture}[scale=1.1]
    \draw[black!30] (-0.23,-0.23) grid (4.23,2.23);
        \foreach \x in {0,4} {
            \foreach \y in {0,2} {
                \draw (\x,\y) circle (0.7mm) [fill=black];
            }
        }
    \draw (0,0) rectangle (1,1);
\draw (0,0) circle (0.7mm) [fill=red];
\node[below] at (2,-0.2) {$\widetilde{\Lambda}$};
\end{tikzpicture}
\hspace{5mm}
\begin{tikzpicture}[scale=1.1]
\draw (0,0) rectangle (1,1) [fill=blue!30];
    \draw[black!30] (-0.23,-0.23) grid (1.23,1.23);
        \foreach \x in {0,0.5,1} {
            \foreach \y in {0,0.25,0.5,0.75,1} {
                \draw (\x,\y) circle (0.7mm) [fill=black];
            }
        }
    \draw (0,0) rectangle (1,1);
\node[below] at (0.6,-0.2) {$\widetilde{\Lambda}^{\circ}$};
\end{tikzpicture}
\end{center}
\end{minipage}
\begin{minipage}{0.5\columnwidth}
\noindent
(ii) If $(a,b) = (1,0)$ or $(3,0)$, then \\
$\widetilde{\Lambda} = \mathbb{Z} \times 2 \mathbb{Z}$,
$\widetilde{\Lambda}^{\circ} = \tfrac{1}{2} \mathbb{Z} \times \mathbb{Z}$.
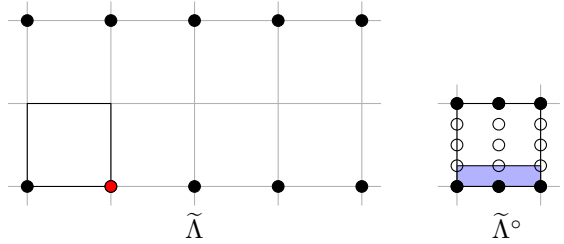
\begin{center}
\begin{tikzpicture}[scale=1.1]
    \draw[black!30] (-0.23,-0.23) grid (4.23,2.23);
        \foreach \x in {0,1,2,3,4} {
            \foreach \y in {0,2} {
                \draw (\x,\y) circle (0.7mm) [fill=black];
            }
        }
    \draw (0,0) rectangle (1,1);
\draw (1,0) circle (0.7mm) [fill=red];
\node[below] at (2,-0.2) {$\widetilde{\Lambda}$};
\end{tikzpicture}
\hspace{5mm}
\begin{tikzpicture}[scale=1.1]
\draw (0,0) rectangle (1,0.25) [fill=blue!30];
    \draw[black!30] (-0.23,-0.23) grid (1.23,1.23);
        \foreach \x in {0,0.5,1} {
            \foreach \y in {0,0.25,0.5,0.75,1} {
                \draw (\x,\y) circle (0.7mm) ;
            }
        }
        \foreach \x in {0,0.5,1} {
            \foreach \y in {0,1} {
                \draw (\x,\y) circle (0.7mm) [fill=black];
            }
        }
    \draw (0,0) rectangle (1,1);
\node[below] at (0.6,-0.2) {$\widetilde{\Lambda}^{\circ}$};
\end{tikzpicture}
\end{center}
\end{minipage}

\noindent
\begin{minipage}{0.5\columnwidth}
\noindent
(iii) If $(a,b) = (2,0)$, then \\
$\widetilde{\Lambda} = 2 \mathbb{Z} \times 2 \mathbb{Z}$, $\widetilde{\Lambda}^{\circ} = \tfrac{1}{2} \mathbb{Z} \times \tfrac{1}{2} \mathbb{Z}$.
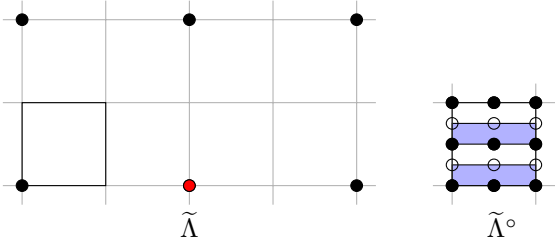
\begin{center}
\begin{tikzpicture}[scale=1.1]
    \draw[black!30] (-0.23,-0.23) grid (4.23,2.23);
        \foreach \x in {0,2,4} {
            \foreach \y in {0,2} {
                \draw (\x,\y) circle (0.7mm) [fill=black];
            }
        }
    \draw (0,0) rectangle (1,1);
\node[below] at (2,-0.2) {$\widetilde{\Lambda}$};
\draw (2,0) circle (0.7mm) [fill=red];
\end{tikzpicture}
\hspace{5mm}
\begin{tikzpicture}[scale=1.1]
\draw (0,0) rectangle (1,0.25) [fill=blue!30];
\draw (0,0.5) rectangle (1,0.75) [fill=blue!30];
    \draw[black!30] (-0.23,-0.23) grid (1.23,1.23);
        \foreach \x in {0,0.5,1} {
            \foreach \y in {0,0.25,0.5,0.75,1} {
                \draw (\x,\y) circle (0.7mm) ;
            }
        }
        \foreach \x in {0,0.5,1} {
            \foreach \y in {0,0.5,1} {
                \draw (\x,\y) circle (0.7mm) [fill=black];
            }
        }
    \draw (0,0) rectangle (1,1);
\node[below] at (0.6,-0.2) {$\widetilde{\Lambda}^{\circ}$};
\end{tikzpicture}
\end{center}
\end{minipage}
\begin{minipage}{0.5\columnwidth}
\noindent
(iv) If $(a,b) = (0,1)$, then \\
$\widetilde{\Lambda} = 4 \mathbb{Z} \times \mathbb{Z}$, $\widetilde{\Lambda}^{\circ} = \mathbb{Z} \times \tfrac{1}{4} \mathbb{Z}$.
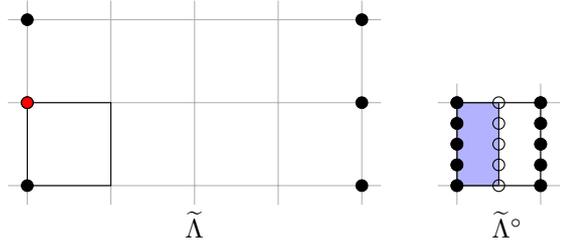
\begin{center}
\begin{tikzpicture}[scale=1.1]
    \draw[black!30] (-0.23,-0.23) grid (4.23,2.23);
        \foreach \x in {0,4} {
            \foreach \y in {0,1,2} {
                \draw (\x,\y) circle (0.7mm) [fill=black];
            }
        }
    \draw (0,0) rectangle (1,1);
\draw (0,1) circle (0.7mm) [fill=red];
\node[below] at (2,-0.2) {$\widetilde{\Lambda}$};
\end{tikzpicture}
\hspace{5mm}
\begin{tikzpicture}[scale=1.1]
\draw (0,0) rectangle (0.5,1) [fill=blue!30];
    \draw[black!30] (-0.23,-0.23) grid (1.23,1.23);
        \foreach \x in {0,0.5,1} {
            \foreach \y in {0,0.25,0.5,0.75,1} {
                \draw (\x,\y) circle (0.7mm) ;
            }
        }
        \foreach \x in {0,1} {
            \foreach \y in {0,0.25,0.5,0.75,1} {
                \draw (\x,\y) circle (0.7mm) [fill=black];
            }
        }
    \draw (0,0) rectangle (1,1);
\node[below] at (0.6,-0.2) {$\widetilde{\Lambda}^{\circ}$};
\end{tikzpicture}
\end{center}
\end{minipage}

\noindent
\begin{minipage}{0.5\columnwidth}
\noindent
(v) $(a,b) = (1,1)$ or $(3,1)$, then \\
$\widetilde{\Lambda} = \left( {2 \atop 0} \, {1 \atop 1} \right) \mathbb{Z}^2$, $\widetilde{\Lambda}^{\circ} = \tfrac{1}{2} \left( {1 \atop 1} \, {2 \atop 0} \right) \mathbb{Z}^2$ (cf.~(\ref{eqn_adjoint_lattice_gen_matrix_form})).
\begin{center}
\begin{tikzpicture}[scale=1.1]
    \draw[black!30] (-0.23,-0.23) grid (4.23,2.23);
        \foreach \x in {0,2,4} {
            \foreach \y in {0,2} {
                \draw (\x,\y) circle (0.7mm) [fill=black];
            }
        }
        \foreach \x in {1,3} {
            \foreach \y in {1} {
                \draw (\x,\y) circle (0.7mm) [fill=black];
            }
        }
    \draw (0,0) rectangle (1,1);
\node[below] at (2,-0.2) {$\widetilde{\Lambda}$};
\draw (1,1) circle (0.7mm) [fill=red];
\end{tikzpicture}
\hspace{5mm}
\begin{tikzpicture}[scale=1.1]
\draw (0,0) rectangle (0.5,0.25) [fill=blue!30];
\draw (0.5,0.5) rectangle (1,0.75) [fill=blue!30];
    \draw[black!30] (-0.23,-0.23) grid (1.23,1.23);
        \foreach \x in {0,0.5,1} {
            \foreach \y in {0,0.25,0.5,0.75,1} {
                \draw (\x,\y) circle (0.7mm) ;
            }
        }
        \foreach \x in {0,1} {
            \foreach \y in {0,1} {
                \draw (\x,\y) circle (0.7mm) [fill=black];
            }
        }
        \foreach \x in {0.5} {
            \foreach \y in {0.5} {
                \draw (\x,\y) circle (0.7mm) [fill=black];
            }
        }
    \draw (0,0) rectangle (1,1);
\node[below] at (0.6,-0.2) {$\widetilde{\Lambda}^{\circ}$};
\end{tikzpicture}
\end{center}
\end{minipage}
\begin{minipage}{0.5\columnwidth}
\noindent
(vi) If $(a,b) = (2,1)$, then \\
$\widetilde{\Lambda} = \left( {4 \atop 0} \, {2 \atop 1} \right) \mathbb{Z}^2$, $\widetilde{\Lambda}^{\circ} = \tfrac{1}{4} \left( {2 \atop 1} \, {4 \atop 0} \right) \mathbb{Z}^2$ (cf.~(\ref{eqn_adjoint_lattice_gen_matrix_form})).
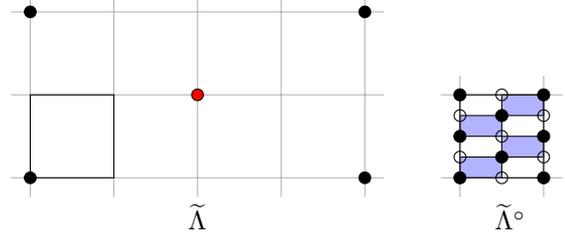
\begin{center}
\begin{tikzpicture}[scale=1.1]
    \draw[black!30] (-0.23,-0.23) grid (4.23,2.23);
        \foreach \x in {0,4} {
            \foreach \y in {0,2} {
                \draw (\x,\y) circle (0.7mm) [fill=black];
            }
        }
    \draw (0,0) rectangle (1,1);
\node[below] at (2,-0.2) {$\widetilde{\Lambda}$};
\draw (2,1) circle (0.7mm) [fill=red];
\end{tikzpicture}
\hspace{5mm}
\begin{tikzpicture}[scale=1.1]
\draw (0,0) rectangle (0.5,0.25) [fill=blue!30];
\draw (0,0.5) rectangle (0.5,0.75) [fill=blue!30];
\draw (0.5,0.25) rectangle (1,0.5) [fill=blue!30];
\draw (0.5,0.75) rectangle (1,1) [fill=blue!30];
    \draw[black!30] (-0.23,-0.23) grid (1.23,1.23);
        \foreach \x in {0,0.5,1} {
            \foreach \y in {0,0.25,0.5,0.75,1} {
                \draw (\x,\y) circle (0.7mm) ;
            }
        }
        \foreach \x in {0,1} {
            \foreach \y in {0,0.5,1} {
                \draw (\x,\y) circle (0.7mm) [fill=black];
            }
        }
        \foreach \x in {0.5} {
            \foreach \y in {0.25,0.75} {
                \draw (\x,\y) circle (0.7mm) [fill=black];
            }
        }
    \draw (0,0) rectangle (1,1);
\node[below] at (0.6,-0.2) {$\widetilde{\Lambda}^{\circ}$};
\end{tikzpicture}
\end{center}
\end{minipage}

\noindent
From Theorem \ref{thm_Gabor_integer_lattice}, we have that $\mathcal{G}(\varphi, 4 \mathbb{Z} \times 2 \mathbb{Z})$ is $\widetilde{\Lambda}$-invariant if and only if for a.e.~$(x,\omega)$, $Z \varphi (x , \omega ) \neq 0$ implies that
\begin{align*}
Z \varphi ( x + \tfrac{r}{2} , \omega + \tfrac{s}{4} ) = 0
\quad \text{for all} \; ( \tfrac{r}{2} , \tfrac{s}{4} ) \in (\tfrac{1}{2} \mathbb{Z} \times \tfrac{1}{4} \mathbb{Z}) \backslash \widetilde{\Lambda}^{\circ} .
\end{align*}
Moreover in this case, if $Z \varphi (x , \omega ) \neq 0$ a.e.~on $B^{(0)} \cap [0,1)^2$, then it follows that $Z \varphi (x , \omega ) = 0$ a.e.~on $[0,1)^2 \backslash B^{(0)}$.

Next, let $\widetilde{\Lambda} \subset \mathbb{R}^2$ be any lattice such that $4 \mathbb{Z} \times 2 \mathbb{Z} \subseteq \widetilde{\Lambda} \subseteq \mathbb{Z}^2$.
There are two cases which are not treated above: $\widetilde{\Lambda} = 2 \mathbb{Z} \times \mathbb{Z}$ and $\widetilde{\Lambda} = \mathbb{Z}^2$.

\medskip

\noindent
\begin{minipage}{0.5\columnwidth}
\noindent
(vii) $\widetilde{\Lambda} = 2 \mathbb{Z} \times \mathbb{Z}$, $\widetilde{\Lambda}^{\circ} = \mathbb{Z} \times \tfrac{1}{2} \mathbb{Z}$.
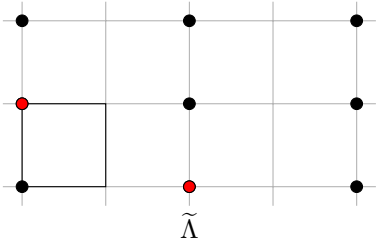
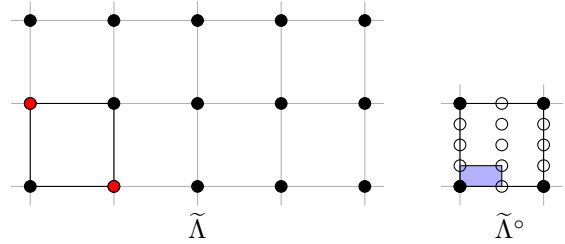
\begin{center}
\begin{tikzpicture}[scale=1.1]
    \draw[black!30] (-0.23,-0.23) grid (4.23,2.23);
        \foreach \x in {0,2,4} {
            \foreach \y in {0,1,2} {
                \draw (\x,\y) circle (0.7mm) [fill=black];
            }
        }
    \draw (0,0) rectangle (1,1);
\node[below] at (2,-0.2) {$\widetilde{\Lambda}$};
\draw (0,1) circle (0.7mm) [fill=red];
\draw (2,0) circle (0.7mm) [fill=red];
\end{tikzpicture}
\hspace{5mm}
\begin{tikzpicture}[scale=1.1]
\draw (0,0) rectangle (0.5,0.25) [fill=blue!30];
\draw (0,0.5) rectangle (0.5,0.75) [fill=blue!30];
    \draw[black!30] (-0.23,-0.23) grid (1.23,1.23);
        \foreach \x in {0,0.5,1} {
            \foreach \y in {0,0.25,0.5,0.75,1} {
                \draw (\x,\y) circle (0.7mm) ;
            }
        }
        \foreach \x in {0,1} {
            \foreach \y in {0,0.5,1} {
                \draw (\x,\y) circle (0.7mm) [fill=black];
            }
        }
    \draw (0,0) rectangle (1,1);
\node[below] at (0.6,-0.2) {$\widetilde{\Lambda}^{\circ}$};
\end{tikzpicture}
\end{center}
\end{minipage}
\begin{minipage}{0.5\columnwidth}
\noindent
(viii) $\widetilde{\Lambda} = \widetilde{\Lambda}^{\circ} = \mathbb{Z}^2$.
\begin{center}
\begin{tikzpicture}[scale=1.1]
    \draw[black!30] (-0.23,-0.23) grid (4.23,2.23);
        \foreach \x in {0,1,2,3,4} {
            \foreach \y in {0,1,2} {
                \draw (\x,\y) circle (0.7mm) [fill=black];
            }
        }
    \draw (0,0) rectangle (1,1);
\node[below] at (2,-0.2) {$\widetilde{\Lambda}$};
\draw (0,1) circle (0.7mm) [fill=red];
\draw (1,0) circle (0.7mm) [fill=red];
\end{tikzpicture}
\hspace{5mm}
\begin{tikzpicture}[scale=1.1]
\draw (0,0) rectangle (0.5,0.25) [fill=blue!30];
    \draw[black!30] (-0.23,-0.23) grid (1.23,1.23);
        \foreach \x in {0,0.5,1} {
            \foreach \y in {0,0.25,0.5,0.75,1} {
                \draw (\x,\y) circle (0.7mm) ;
            }
        }
        \foreach \x in {0,1} {
            \foreach \y in {0,1} {
                \draw (\x,\y) circle (0.7mm) [fill=black];
            }
        }
    \draw (0,0) rectangle (1,1);
\node[below] at (0.6,-0.2) {$\widetilde{\Lambda}^{\circ}$};
\end{tikzpicture}
\end{center}
\end{minipage}
When $\widetilde{\Lambda} = \mathbb{Z}^2$, Corollary \ref{cor_Gabor_integer_lattice} states that $\mathcal{G}(\varphi, 4 \mathbb{Z} \times 2 \mathbb{Z})$ is $\mathbb{Z}^2$-invariant if and only if there exists a fundamental domain $D \subset [0,1)^{2}$ of the lattice $\tfrac{1}{2} \mathbb{Z} \times \tfrac{1}{4} \mathbb{Z}$ such that $Z \varphi (x , \omega ) = 0$ a.e.~on $[0,1)^{2} \backslash D$.
For example, as depicted in (viii), if $Z \varphi |_{[0,1)^2}$ is supported on $[0,\tfrac{1}{2}) \times [0,\tfrac{1}{4})$, equivalently, if $Z \varphi (x , \omega )$ is supported on $B^{(0)}$, then $\mathcal{G}(\varphi, 4 \mathbb{Z} \times 2 \mathbb{Z})$ is $\mathbb{Z}^2$-invariant.
\end{example}

\medskip

\begin{remark}[Shift-invariant spaces vs.~Gabor spaces --- Remarks \ref{rmk_SIS_extra_inv} and \ref{rmk_Gabor_extra_inv} revisited]
\label{rmk_revisit_SIS_Gabor_rmks}
\rm
$ $\\
\noindent
(i) Extra invariance of shift-invariant spaces \\
Let $\varphi \in L^2(\mathbb{R}^{2d})$ and let $\widetilde{\Gamma} \subset \mathbb{R}^{2d}$ be a proper super-lattice of $\mathbb{Z}^{2d}$, i.e., a lattice strictly containing $\mathbb{Z}^{2d}$. Then $\mathbb{Z}^{2d} \subsetneq \widetilde{\Gamma} \subseteq \frac{1}{m} \mathbb{Z}^d \times \frac{1}{n} \mathbb{Z}^d$ for some $m,n \in \mathbb{N}$ (cf.~Section \ref{subsec:SIS_extra_inv}),
Assume that $\mathcal{S}(\varphi, \mathbb{Z}^{2d})$ is invariant under translations by $\widetilde{\Gamma}$, that is, $\mathcal{S} (\varphi, \mathbb{Z}^{2d}) = \mathcal{S} (\varphi, \widetilde{\Gamma})$.
Then every function $f \in \mathcal{S} (\varphi, \mathbb{Z}^{2d})$ can be expressed in two different ways:
\begin{align*}
m(\xi, \omega) \, \widehat{\varphi}(\xi, \omega) = \widehat{f}(\xi, \omega) = \widetilde{m}(\xi, \omega) \, \widehat{\varphi}(\xi, \omega) ,
\end{align*}
where $m(\xi,\omega)$ is $\mathbb{Z}^{2d}$-periodic and $\widetilde{m}(\xi,\omega)$ is $\widetilde{\Gamma}^*$-periodic. Since the Fourier transforms $\widehat{\varphi}(\xi, \omega) , \widehat{f}(\xi, \omega) \in L^2(\mathbb{R}^{2d})$ are non-periodic, there is no other periodicity involved in the equation.
From the different periodicity of $m(\xi, \omega)$ and $\widetilde{m}(\xi, \omega)$, we get some constraints on the set of zeros of $\widehat{\varphi}(\xi, \omega)$.
Note that $\mathbb{Z}^{2d} \subsetneq \widetilde{\Gamma} \subseteq \frac{1}{m} \mathbb{Z}^d \times \frac{1}{n} \mathbb{Z}^d$ implies $m \mathbb{Z}^d \times n \mathbb{Z}^d \subseteq \widetilde{\Gamma}^* \subsetneq \mathbb{Z}^{2d}$.
With $m,n \in \mathbb{N}$ large, the lattice $\widetilde{\Gamma}$ has a large density which corresponds to $\widetilde{\Gamma}^*$ having a small density. \\[2mm]
(ii) Extra invariance of Gabor spaces \\
Let $\varphi \in L^2(\mathbb{R}^{2d})$ and let $\Lambda, \widetilde{\Lambda} \subseteq \mathbb{R}^{2d}$ be lattices satisfying $\Lambda \subsetneq \widetilde{\Lambda} \subseteq \mathbb{Z}^{2d}$ (so that $\mathbb{Z}^{2d} \subseteq \widetilde{\Lambda}^{\circ} \subsetneq {\Lambda}^{\circ}$).
Assume that $\mathcal{G}(\varphi, \Lambda)$ is $\widetilde{\Lambda}$-invariant, that is, $\mathcal{G}(\varphi, \Lambda) = \mathcal{G}(\varphi, \widetilde{\Lambda})$. Then every function $f$ in $\mathcal{G}(\varphi, \Lambda)$ can be expressed in two different ways:
$$
h ( x , \omega ) \,  Z \varphi ( x , \omega ) = Z f ( x , \omega ) = \widetilde{h} ( x , \omega ) \,  Z \varphi ( x , \omega ),
$$
where $h ( x , \omega )$ is ${\Lambda}^{\circ}$-periodic and $\widetilde{h} ( x , \omega )$ is $\widetilde{\Lambda}^{\circ}$-periodic.
Note that unlike the (non-periodic) Fourier transform, the Zak transform is quasi-periodic. Therefore, by replacement if necessary, $h ( x , \omega )$ and $\widetilde{h} ( x , \omega )$ can be assumed to be $\mathbb{Z}^{2d}$-periodic, which conforms to the condition that $\Lambda, \widetilde{\Lambda} \subseteq \mathbb{Z}^{2d}$ (equivalently, ${\Lambda}^{\circ}, \widetilde{\Lambda}^{\circ} \supseteq \mathbb{Z}^{2d}$).
This obviously limits the lattice $\widetilde{\Lambda}$ to be coarser than $\mathbb{Z}^{2d}$, whereas the lattice $\widetilde{\Gamma}$ discussed in (i) has no such restrictions and can be highly dense.
Similarly as in (i), we obtain some constraints on the set of zeros of $Z \varphi ( x , \omega )$ using the different periodicity of $h ( x , \omega )$ and $\widetilde{h} ( x , \omega )$.
%%% Similarly, some restrictions on the zeros of $Z \varphi ( x , \omega )$ are obtained from the different periodicity of $h ( x , \omega )$ and $\widetilde{h} ( x , \omega )$.

\medskip

%%% {\footnotesize
%%% {\normalsize
{\small
\begin{center}
\renewcommand{\arraystretch}{2}
\begin{tabular}{| c | c | c | c | c | c |}
\hline
Space & Invariance Lattice & Periodicity of Transform \\[-.3cm]
 & (Dual/Adjoint Lattice) & \\
\hline
$\mathcal{S} (\varphi, \mathbb{Z}^{2d})$ & $\mathbb{Z}^{2d} \subsetneq \widetilde{\Gamma} \subseteq \tfrac{1}{m} \mathbb{Z}^d \times \tfrac{1}{n} \mathbb{Z}^d$ &
$\widehat{\varphi}(\xi, \omega)$ is non-periodic \\[-.3cm]
 &  ($m \mathbb{Z}^d \times n \mathbb{Z}^d \subseteq \widetilde{\Gamma}^* \subsetneq \mathbb{Z}^{2d}$) & \\ \hline
$\mathcal{G}(\varphi, \Lambda)$ & $\Lambda \subsetneq \widetilde{\Lambda} \subseteq \mathbb{Z}^{2d}$ &
$Z \varphi ( x , \omega )$ is quasi-periodic \\[-.3cm]
&  ($\mathbb{Z}^{2d} \subseteq \widetilde{\Lambda}^{\circ} \subsetneq {\Lambda}^{\circ}$) & \\ \hline
\end{tabular}
\end{center}
}
\end{remark}

\medskip

\newcommand{\etalchar}[1]{$^{#1}$}
\def\cprime{$'$} \def\cprime{$'$}
\providecommand{\bysame}{\leavevmode\hbox to3em{\hrulefill}\thinspace}
\providecommand{\MR}{\relax\ifhmode\unskip\space\fi MR }
\providecommand{\MRhref}[2]{%
  \href{http://www.ams.org/mathscinet-getitem?mr=#1}{#2}
}
\providecommand{\href}[2]{#2}

%\bibliographystyle{amsalpha}
%\bibliography{cyu,gg}
%

\medskip

\renewcommand{\thetheorem}{A.\arabic{theorem}}
\setcounter{theorem}{0}

\section*{Appendix I}
\label{sec:appendix_I}

%%% [cf.~Proposition 2.1 in \cite{ACP11}]
\begin{proposition-app}
\label{prop_M_is_closed}
Let $V$ be a closed subspace of $L^2(\mathbb{R}^d)$ and $\Lambda \subset \mathbb{R}^{2d}$ be a lattice.
If $V$ is $\Lambda$-invariant, then $\mathcal{P} (V)$ is an additive closed subgroup of $\mathbb{R}^{2d}$ containing $\Lambda$.
\end{proposition-app}

To prove Proposition \ref{prop_M_is_closed}, we will use similar arguments as in the proof of \cite[Proposition 2.1]{ACP11}.
Recall that an additive semigroup is a nonempty set with an associative additive operation.
We need the following lemma which is proven for the case $\Gamma = \mathbb{Z}^d$ in \cite[Lemma 2.2]{ACP11}.

%%% [cf.~Lemma 2.2 in \cite{ACP11}]
\begin{lemma-app}
\label{lem_semigp_is_gp}
Let $\Gamma \subset \mathbb{R}^d$ be a lattice. If $H$ is a closed additive semigroup of $\mathbb{R}^d$ containing $\Gamma$, then $H$ is an additive group.
\end{lemma-app}

\begin{proof}
Let $\Pi$ be the quotient map from $\mathbb{R}^d$ onto $D = \mathbb{R}^d / \Gamma$. Here $D$ is a fundamental domain of the lattice $\Gamma$. Since $H$ is a semigroup containing $\Gamma$, we have $H + \Gamma = H$ where $H + \Gamma$ denotes the set $\{ h + \gamma \, : \, h \in H, ~ \gamma \in \Gamma \}$.
Indeed, $H + \Gamma \subseteq H$ comes from the fact that $H$ is closed under addition, and $H \subseteq H + \Gamma$ is due to the fact that $0 \in \Gamma$. Therefore,
\begin{align*}
\Pi^{-1} [ \Pi (H) ] &= \bigcup_{h \in \Pi(H)} h + \Gamma = \bigcup_{h \in H} h + \Gamma = H + \Gamma = H .
\end{align*}
This implies that $\Pi (H)$ is closed in $D$ and is therefore compact.

Since a compact semigroup of $D$ is necessarily a group \cite[Theorem 9.16]{HR63}, it follows that $\Pi (H) \subset D$ is a group and consequently $H$ is a group.
\end{proof}

\noindent \textit{Proof of Proposition \ref{prop_M_is_closed}.}\;
It is immediate from definition that $\Lambda \subseteq \mathcal{P} (V)$. To show that $\mathcal{P} (V)$ is closed, let $\{ (u_n,\eta_n) \}_{n=1}^{\infty} \subset \mathcal{P} (V)$ and $(u_0,\eta_0) \in \mathbb{R}^d \times \mathbb{R}^d$ be such that $\lim_{n \rightarrow} (u_n,\eta_n) = (u_0,\eta_0)$ in the usual product topology of $\mathbb{R}^d \times \mathbb{R}^d$. Then for every $f \in \mathcal{G}(\varphi, \Lambda)$, we have
\begin{align*}
\lefteqn{\| \pi(u_n,\eta_n) f - \pi(u_0,\eta_0) f \|_2
= \| M_{\eta_n} T_{u_n} f - M_{\eta_0} T_{u_0} f \|_2 } \\
&\leq
\| M_{\eta_n} ( T_{u_n} f - T_{u_0} f ) \|_2
+ \| (M_{\eta_n} - M_{\eta_0}) T_{u_0} f \|_2 \\
&=
\left( \int_{\mathbb{R}^d}  | f(x - u_n) - f(x - u_0) |^2  \, dx \right)^{1/2}
+ \left(\int_{\mathbb{R}^d}  | e^{2\pi i \eta_n \cdot x} - e^{2\pi i \eta_0 \cdot x} |^2 \cdot | f(x - u_0) |^2  \, dx \right)^{1/2} \\
&\rightarrow 0 \quad \text{as} \;\; n \rightarrow 0 ,
\end{align*}
which implies that $\pi(u_0,\eta_0) f \in \overline{V} = V$.
Therefore, $\mathcal{P} (V)$ is closed.

Next, we show that $\mathcal{P} (V)$ is a semigroup of $\mathbb{R}^{2d}$. Let $(u,\eta), (u',\eta') \in \mathcal{P} (V)$. Then for any $f \in V$, we have $\pi(u,\eta) f \in V$ and in turn $\pi(u',\eta') [\pi(u,\eta) f] \in V$.
Noting that $\pi(u + u', \eta + \eta') = e^{2 \pi i \eta \cdot u'} \, \pi(u',\eta') \circ \pi(u,\eta)$ (cf.~(\ref{eqn_commut_rel})), we have $\pi(u + u', \eta + \eta') f = e^{2 \pi i \eta \cdot u'} \cdot \pi(u',\eta') [\pi(u,\eta) f] \in V$, therefore, $(u + u', \eta + \eta') \in \mathcal{P} (V)$.
This shows that $\mathcal{P} (V)$ is closed under the additive operation given by $(u,\eta) + (u',\eta') = (u + u', \eta + \eta')$. It is easy to check that this operation is associative, thus $\mathcal{P} (V)$ is a semigroup of $\mathbb{R}^{2d}$.

Finally, since $\mathcal{P} (V)$ is a closed semigroup of $\mathbb{R}^{2d}$ containing a lattice $\Lambda$, we conclude from Lemma \ref{lem_semigp_is_gp} that $\mathcal{P} (V)$ is a group.
   \hfill $\Box$ %\hfill $\square$ %end of proof

\section*{Appendix II -- Proof of Lemma \ref{lem_Gabor_ftn_exp_lambda_integer}}
\label{sec:appendix_II}

To prove Lemma \ref{lem_Gabor_ftn_exp_lambda_integer}, we will use arguments similar to the proof of \cite[Theorem 2.14]{bdr}.

For $f,g \in L^2(\mathbb{R}^d)$ and a lattice $\Lambda \subseteq \mathbb{Z}^{2d}$, we define the $\Lambda^{\circ}$-periodic function
\begin{align*}
[f,g]_{\Lambda^{\circ}} (x,\omega) := \sum_{(u,\eta) \in {\Lambda}^{\circ} \cap [0,1)^{2d}} Z f (x + u , \omega + \eta) \, \overline{Z g (x + u , \omega + \eta)} .
\end{align*}
It is clear that $[f,f]_{\Lambda^{\circ}} (x,\omega) \geq 0$ and by the Cauchy-Schwarz inequality, $| [f,g]_{\Lambda^{\circ}} (x,\omega)|^2 \leq [f,f]_{\Lambda^{\circ}} (x,\omega) \cdot [g,g]_{\Lambda^{\circ}} (x,\omega)$.

Let $\varphi \in L^2(\mathbb{R}^d)$ and let $\Lambda \subseteq \mathbb{Z}^{2d}$ be a lattice.
We denote by $P = P_{\mathcal{G}(\varphi, \Lambda)}$ the orthogonal projection from $L^2(\mathbb{R}^d)$ onto $\mathcal{G}(\varphi, \Lambda)$, and by $Q = P_{Z[\mathcal{G}(\varphi, \Lambda)]}$ the orthogonal projection from $L^2 ([0,1)^{2d})$ onto $Z[\mathcal{G}(\varphi, \Lambda)]$.
Then $Z P = Q Z$.
Indeed, for any fixed $f \in L^2(\mathbb{R}^d)$, we have
\begin{align*}
\| f - Pf \| \leq \| f - g \|
\quad \text{for all} \;\; g \in \mathcal{G}(\varphi, \Lambda) ,
\end{align*}
and since the Zak transform is unitary, this is equivalent to
\begin{align*}
\| Zf - Z(Pf) \|_{L^2 ([0,1)^{2d})} \leq \| Zf - Zg \|_{L^2 ([0,1)^{2d})}
\quad \text{for all} \;\; Zg \in Z[\mathcal{G}(\varphi, \Lambda)] ,
\end{align*}
By the uniqueness of best approximation in $L^2 ([0,1)^{2d})$, it follows that $Z(Pf) = P_{Z[\mathcal{G}(\varphi, \Lambda)]} (Zf) = Q(Zf)$, and consequently, $Z P = Q Z$.

\begin{proposition-app}
\label{prop_append_II_orthog_proj}
Let $\varphi \in L^2(\mathbb{R}^d)$ and let $\Lambda \subseteq \mathbb{Z}^{2d}$ be a lattice.
For any $f \in L^2(\mathbb{R}^d)$, we have $Z(Pf) (x,\omega) = h_f (x,\omega) Z \varphi (x,\omega)$,
where
\begin{align*}
h_f (x,\omega) :=
\left\{
   \begin{array}{ll}
   [f,\varphi]_{\Lambda^{\circ}} (x,\omega) /[\varphi,\varphi]_{\Lambda^{\circ}} (x,\omega) \quad &
   \text{on } \;\; \supp [\varphi,\varphi]_{\Lambda^{\circ}} , \\
   0  \quad & \text{otherwise} .
   \end{array}
   \right.
\end{align*}
\end{proposition-app}

\begin{proof}
Note that $h_f (x,\omega) Z \varphi (x,\omega) \in L^2 ([0,1)^{2d})$ for any $f \in L^2(\mathbb{R}^d)$. In fact, $\Psi :  L^2 ([0,1)^{2d}) \rightarrow L^2 ([0,1)^{2d})$ defined by $\Psi (Zf) (x,\omega) = h_f (x,\omega) Z \varphi (x,\omega)$, $f \in L^2(\mathbb{R}^d)$, is a bounded linear operator with $\| \Psi \| \leq 1$.
To see this, we compute
\begin{align*}
\| h_f Z \varphi \|_{L^2 ([0,1)^{2d})}^2
&= \iint_{[0,1)^{2d}} |h_f (x,\omega)|^2 Z \varphi (x,\omega)|^2 \, dx d\omega
= \iint_{[0,1)^{2d}} \left|
\frac{[f,\varphi]_{\Lambda^{\circ}} (x,\omega)  }{[\varphi,\varphi]_{\Lambda^{\circ}} (x,\omega)  } \right|^2  |Z \varphi (x,\omega)|^2 \, dx d\omega \\
&= \iint_{D} \left|
\frac{[f,\varphi]_{\Lambda^{\circ}} (x,\omega)  }{[\varphi,\varphi]_{\Lambda^{\circ}} (x,\omega)  } \right|^2 \sum_{(u,\eta) \in {\Lambda}^{\circ} \cap [0,1)^{2d}} |Z \varphi (x + u , \omega + \eta)|^2 \, dx d\omega \\
&= \iint_{D}
\frac{|[f,\varphi]_{\Lambda^{\circ}} (x,\omega) |^2 }{[\varphi,\varphi]_{\Lambda^{\circ}} (x,\omega)  } \chi_{\supp [\varphi,\varphi]_{\Lambda^{\circ}}} (x,\omega) \, dx d\omega \\
&\leq \iint_{D}
[f,f]_{\Lambda^{\circ}} (x,\omega) \, dx d\omega
= \iint_{[0,1)^{2d}} |Z f (x,\omega)|^2 \, dx d\omega
= \| Z f \|_{L^2 ([0,1)^{2d})}^2 ,
\end{align*}
where we have used the fact that $\supp Z \varphi \subseteq \supp [\varphi,\varphi]_{\Lambda^{\circ}}$. Here $D \subset \mathbb{R}^{2d}$ is a fundamental domain of the lattice $\Lambda^{\circ}$.

We will show that $\Psi$ is the orthogonal projection from $L^2 ([0,1)^{2d})$ onto $Z[\mathcal{G}(\varphi, \Lambda)]$, i.e., $\Psi = Q$. It then follows immediately that $Z(Pf) (x,\omega) = Q (Zf) (x,\omega) = \Psi (Zf) (x,\omega) = h_f (x,\omega) Z \varphi (x,\omega)$ for any $f \in L^2(\mathbb{R}^d)$.

First, we verify that $\Psi = 0$ on $Z[\mathcal{G}(\varphi, \Lambda)]^{\perp}$.
Since the Zak transform is unitary, we have $Z[\mathcal{G}(\varphi, \Lambda)]^{\perp} = Z[\mathcal{G}(\varphi, \Lambda)^{\perp}]$ so it is enough to check that $h_f (x,\omega) Z \varphi (x,\omega) = 0$ for all $f \in \mathcal{G}(\varphi, \Lambda)^{\perp}$.
Using (\ref{Zak_on_pi_quasi_peri}), we obtain
\begin{align*}
f \in \mathcal{G}(\varphi, \Lambda)^{\perp}
&\Leftrightarrow \;\; f \perp \pi(u,\eta) \varphi \quad \text{for all} \;\; (u,\eta) \in \Lambda \\
&\Leftrightarrow \;\; Zf \perp Z(\pi(u,\eta) \varphi) \quad \text{for all} \;\; (u,\eta) \in \Lambda \\
&\Leftrightarrow \;\; 0 = \iint_{[0,1)^{2d}}
Zf (x,\omega) \overline{Z \varphi ( x , \omega )} \, e^{- 2 \pi i (\eta \cdot x - u \cdot \omega)} \, dx d\omega \quad \text{for all} \;\; (u,\eta) \in \Lambda \\
&\Leftrightarrow \;\; 0 = \iint_{D}
[f,\varphi]_{\Lambda^{\circ}} (x,\omega) \, e^{- 2 \pi i (\eta \cdot x - u \cdot \omega)} \, dx d\omega \quad \text{for all} \;\; (u,\eta) \in \Lambda \\
&\Leftrightarrow \;\;
[f,\varphi]_{\Lambda^{\circ}} (x,\omega) = 0
\quad \text{a.e.~on} \; D , \;\; \text{(thus, a.e.~on $\mathbb{R}^{2d}$)},
\end{align*}
where we have used the fact that $\{ e^{- 2 \pi i (\eta \cdot x - u \cdot \omega)} \}_{(u,\eta) \in \Lambda}$ is a Fourier basis for $L^2(D)$. This shows that for $f \in \mathcal{G}(\varphi, \Lambda)^{\perp}$, we have $h_f (x,\omega) = 0$ and therefore $h_f (x,\omega) Z \varphi (x,\omega) = 0$.

Next, in order to prove that $\Psi = id$ on $Z[\mathcal{G}(\varphi, \Lambda)]$, it is enough to show that $\Psi ( Z (\pi(u,\eta) \varphi)) = Z (\pi(u,\eta) \varphi)$ for all $(u,\eta) \in \Lambda$. Let us fix any $(u,\eta) \in \Lambda$.
Then for $(x,\omega) \in \supp [\varphi,\varphi]_{\Lambda^{\circ}}$,
\begin{align*}
h_{Z (\pi(u,\eta) \varphi)} (x,\omega)
&=
\frac{[\pi(u,\eta) \varphi,\varphi]_{\Lambda^{\circ}} (x,\omega) }{[\varphi,\varphi]_{\Lambda^{\circ}} (x,\omega)} \\
&=
\frac{\sum_{(u',\eta') \in {\Lambda}^{\circ} \cap [0,1)^{2d}}
e^{2 \pi i [\eta \cdot (x + u') - u \cdot (\omega + \eta')]} Z \varphi (x + u' , \omega + \eta') \, \overline{Z \varphi (x + u' , \omega + \eta')} }{[\varphi,\varphi]_{\Lambda^{\circ}} (x,\omega)} \\
&=
\frac{e^{- 2 \pi i (\eta \cdot x - u \cdot \omega)} \, [\varphi,\varphi]_{\Lambda^{\circ}} (x,\omega) }{[\varphi,\varphi]_{\Lambda^{\circ}} (x,\omega)} \\
&=
e^{- 2 \pi i (\eta \cdot x - u \cdot \omega)} ,
\end{align*}
where we have used (\ref{Zak_on_pi_quasi_peri}) and the fact that $e^{2 \pi i (\eta \cdot u' - u \cdot \eta')} = 1$ for all $(u',\eta') \in {\Lambda}^{\circ}$.
Since $\supp Z \varphi \subseteq \supp [\varphi,\varphi]_{\Lambda^{\circ}}$, we obtain
\begin{align*}
h_{Z (\pi(u,\eta) \varphi)} (x,\omega) Z \varphi (x,\omega)
= e^{- 2 \pi i (\eta \cdot x - u \cdot \omega)} Z \varphi (x,\omega) = Z (\pi(u,\eta) \varphi) (x,\omega),
\end{align*}
which means that $\Psi ( Z (\pi(u,\eta) \varphi)) = Z (\pi(u,\eta) \varphi)$. This completes the proof.
\end{proof}

\noindent \textit{Proof of Lemma \ref{lem_Gabor_ftn_exp_lambda_integer}.}\;
If $f \in \mathcal{G}(\varphi, \Lambda)$, then $Pf = f$ so that Proposition \ref{prop_append_II_orthog_proj} gives $Zf (x,\omega) = h_f (x,\omega) Z \varphi (x,\omega)$, where $h_f (x,\omega)$ is ${\Lambda}^{\circ}$-periodic.
Conversely, assume that $f \in L^2(\mathbb{R}^d)$ is such that $Z f ( x , \omega ) = h ( x , \omega ) \,  Z \varphi ( x , \omega )$ for some ${\Lambda}^{\circ}$-periodic function $h ( x , \omega )$.
Then
\begin{align*}
[f,\varphi]_{\Lambda^{\circ}} (x,\omega)
&= \sum_{(u,\eta) \in {\Lambda}^{\circ} \cap [0,1)^{2d}} Z f (x + u , \omega + \eta) \, \overline{Z \varphi (x + u , \omega + \eta)}
=  h ( x , \omega ) \, [\varphi,\varphi]_{\Lambda^{\circ}} (x,\omega) ,
\end{align*}
so that $h_f (x,\omega) = h (x,\omega)$ on the support of $[\varphi,\varphi]_{\Lambda^{\circ}}$. Since $\supp Z \varphi \subseteq \supp [\varphi,\varphi]_{\Lambda^{\circ}}$, it follows that $Z(Pf) (x,\omega) = h_f (x,\omega) Z \varphi (x,\omega) = h (x,\omega) Z \varphi (x,\omega) = Z f ( x , \omega )$. Therefore, $Pf = f$ which means that $f \in \mathcal{G}(\varphi, \Lambda)$. This completes the proof.
\hfill $\Box$ %\hfill $\square$ %end of proof

\bigskip

\renewcommand{\thetheorem}{\arabic{theorem}}

\end{document}